\documentclass[letterpaper]{amsart}
\usepackage{amssymb}
\usepackage{amsthm}
\usepackage{amsmath}
\usepackage{amsfonts}
\usepackage{comment}
\usepackage[mathscr]{euscript}
\usepackage{tipa}
\usepackage{graphicx}
\usepackage{xfrac}
\usepackage{MnSymbol}
\usepackage{enumitem}
    
\usepackage{soul}

\usepackage{xargs}
\usepackage{csquotes}

\newcommandx{\change}[2][1=]{\todo[linecolor=blue,backgroundcolor=blue!25,bordercolor=blue,#1]{#2}}

\newcommand{\res}{\mathord{\upharpoonright}} 

\newcommand{\qq}{\mathbb{Q}}

\newcommand{\rr}{\mathbb{R}}

\newcommand{\nn}{\mathbb{N}}

\newcommand{\zz}{\mathbb{Z}}

\renewcommand{\gg}{\mathbb{G}}

\newcommand{\la}{\langle}
\newcommand{\ra}{\rangle}

\newcommand{\sU}{\mathscr{U}}

\newcommand{\sP}{\mathscr{P}}

\newcommand{\SF}{\mathrm{SF}^\zz}
\newcommand{\SQ}{\mathrm{SF}^\qq}
\newcommand{\SFZ}{\mathrm{SF}_{\zz}}
\newcommand{\SFQ}{\mathrm{SF}_{\qq}}
\newcommand{\OSFQ}{\mathrm{OSF}_{\qq}}
\newcommand{\TSFZ}{\mathrm{SF}^*_{\zz}}
\newcommand{\TSFQ}{\mathrm{SF}^*_{\qq}}
\newcommand{\TOSFQ}{\mathrm{OSF}^*_{\qq}}
\newcommand{\WSFZ}{\mathrm{Sf}^*_{\zz}}
\newcommand{\WSFQ}{\mathrm{Sf}^*_{\qq}}
\newcommand{\WOSFQ}{\mathrm{OSf}^*_{\qq}}

\makeatletter
\newsavebox\myboxA
\newsavebox\myboxB
\newlength\mylenA

\newcommand*\xbar[2][0.75]{%
    \sbox{\myboxA}{$\m@th#2$}%
    \setbox\myboxB\null
    \ht\myboxB=\ht\myboxA%
    \dp\myboxB=\dp\myboxA%
    \wd\myboxB=#1\wd\myboxA
    \sbox\myboxB{$\m@th\overline{\copy\myboxB}$}
    \setlength\mylenA{\the\wd\myboxA}
    \addtolength\mylenA{-\the\wd\myboxB}%
    \ifdim\wd\myboxB<\wd\myboxA%
       \rlap{\hskip 0.5\mylenA\usebox\myboxB}{\usebox\myboxA}%
    \else
        \hskip -0.5\mylenA\rlap{\usebox\myboxA}{\hskip 0.5\mylenA\usebox\myboxB}%
    \fi}
\makeatother

\newtheorem{thm}{Theorem}[section]
\newtheorem{lem}[thm]{Lemma}

\newtheorem*{thm*}{Theorem}
\newtheorem*{conj*}{Conjecture}
\newtheorem*{prop*}{Proposition}
\newtheorem*{fact*}{Fact}
\newtheorem{prop}[thm]{Proposition}

\newtheorem{cor}[thm]{Corollary}
\newtheorem{rem}[thm]{Remark}
\newtheorem*{eg*}{Example}

\providecommand{\keywords}[1]
{
  \small	
  \textbf{\textit{Keywords---}} #1
}

\begin{document}
\sloppy


\title[ The groups $\mathbb{Z}$ and $\mathbb{Q}$ with predicates for being square-free]{The additive groups of $\mathbb{Z}$ and $\mathbb{Q}$ with predicates for being square-free}
\author{Neer Bhardwaj, Minh Chieu Tran}
\address{Department of Mathematics, University of Illinois at Urbana-
Champaign, Urbana, IL 61801, U.S.A}
\curraddr{}
\email{nbhard4@illinois.edu}
\address{Department of Mathematics, University of Notre Dame, Notre Dame, IN 46556, U.S.A}
\curraddr{}
\email{mtran6@nd.edu}
\date{\today}

\begin{abstract}
We consider the four structures $(\mathbb{Z}; \SF)$, $(\mathbb{Z}; <, \SF)$, $(\mathbb{Q}; \SQ)$,  and  $(\mathbb{Q}; <, \SQ)$ where $\mathbb{Z}$ is the additive group of integers, $\SF$ is the set of $a \in \mathbb{Z}$ such that $v_{p}(a) < 2$ for every prime $p$ and corresponding $p$-adic valuation $v_{p}$, $\mathbb{Q}$ and $\SQ$ are defined likewise for rational numbers, and $<$ denotes the natural ordering on each of these domains. We prove that the second structure is model-theoretically wild while the other three structures are model-theoretically tame. Moreover, all these results can be seen as examples where number-theoretic randomness yields model-theoretic consequences.
\end{abstract}


\maketitle
\section{Introduction}

\noindent In  \cite{ShelahKaplan}, Kaplan and Shelah showed under the assumption of Dickson's conjecture that if  $\zz$ is the additive group of integers implicitly assumed to contain the element $1$ as a distinguished constant and the map $a \mapsto -a$ as a distinguished function, and if $\text{Pr}$ is the set of $ a \in \zz$  such that either  $ a$ or $-a$ is prime, then the theory of $(\zz; \text{Pr})$  is model complete, decidable, and super-simple of U-rank $1$. From our current point of view, the above result can be seen as an example of a more general phenomenon where we can often capture aspects of randomness inside a structure using first-order logic and deduce in consequence  several model-theoretic properties of that structure. In $(\zz; \text{Pr})$, the conjectural randomness is that of the set of primes with respect to addition. Dickson's conjecture is useful here as it reflects this randomness in a fashion which can be made first-order. The second author's work in \cite{Minh} provides another example with similar themes. 

\medskip \noindent Our viewpoint in particular predicts that there are analogues of Kaplan and Shelah's results with $\textrm{Pr}$ replaced by other random subsets of $\zz$. We confirm the above prediction in this paper without the assumption of any conjecture when $\textrm{Pr}$ is replaced with the set $$\SF = \{ a \in \zz : \text{ for all } p \text{ primes, }  v_p(a) < 2 \}$$ where $v_p$ is the $p$-adic valuation associated to the prime $p$. 
We have that $\zz$ is a structure in the language $L$ of additive groups augmented by  a constant symbol for $1$  and a function symbol for $a \mapsto -a$. Then $(\zz;\SF)$  is a structure in the language $L_{\mathrm{u}}$ extending $L$ by a unary predicate symbol for $\SF$ (as indicated by the additional subscript ``u''). We will introduce a first-order notion of {\it genericity} which captures the partial randomness in the interaction between $\SF$ and the additive structure on $\zz$. Using a similar idea as in \cite{ShelahKaplan},  we obtain:

\begin{thm} \label{structure1}
The theory of $(\zz; \SF)$ is model complete, decidable, supersimple of U-rank $1$, and is $k$-independent for all $k \in \nn^{\geq 1}$.
\end{thm}

\noindent The above theorem gives us without assuming any conjecture the first natural example of a simple unstable expansion of $\zz$. From the same notion of {\it genericity}, we deduce entirely different consequences  for the structure $(\zz; <, \SF)$ in the language $L_\mathrm{ou}$ extending $L_\mathrm{u}$ by a binary predicate symbol for the natural ordering $<$ (as indicated by the additional subscript ``o''):

\begin{thm} \label{structure2}
The theory of $(\zz; <, \SF)$ interprets arithmetic.
\end{thm}
\noindent The proof here is an adaption of the strategy used in \cite{BJW} to show that the theory of $(\nn; +, <, \Pr)$ with $\Pr$ the set of primes interprets arithmetic.  The above two theorems are in stark contrast  with one another in view of the fact that $(\zz;<)$ is a minimal proper expansion of $\zz$; indeed, it is proven in \cite{Gabe} that adding any new definable set from $(\zz;<)$ to $\zz$ results in defining $<$. On the other hand, it is shown in \cite{DolichGoodrick} that there is no strong expansion of the theory of Presburger arithmetic, so  Theorem \ref{structure2} is perhaps not entirely unexpected.  
\medskip

\noindent It is also natural to consider the structures $(\qq; \SQ)$ and  $(\qq;<, \SQ)$ where $\qq$ is the additive group of rational numbers, also implicitly assumed to contain  $1$ as a distinguished constant and  $a\mapsto -a$ as a distinguished function,
$$\SQ =\{ a \in \qq :   v_p(a) < 2 \text{ for all } \text{ primes } p  \},$$
and the relation $<$ on $\qq$ is the natural ordering. The reader might wonder why chose the above $\SQ$  instead of $\SF$ or $\text{ASF}^\qq =\{ a \in \qq :   |v_p(a)| < 2 \text{ for all } \text{ primes } p  \}.$ From Lemma \ref{sumoftwosquarefree} in the next section, we get $\SQ+\SQ= \qq$, $\SF+\SF = \zz$, and $\text{ASF}^\qq + \text{ASF}^\qq=\{ a : v_p(a) >-2 \text{ for all } \text{ primes } p\}$. Hence, equipping $\qq$ and $(\qq;<)$ with either $\SF$ or $\text{ASF}^\qq$ will result in structures 
expanding a infinite-index pair of infinite abelian groups with a unary predicate on the smaller group, and therefore, having rather different flavors from $(\zz;\SF)$ and $(\zz;<, \SF)$. 

\medskip \noindent 
Viewing $(\qq; \SQ)$ and $(\qq;<, \SQ)$ in the obvious way as an $L_\mathrm{u}$-structure and an $L_\mathrm{ou}$-structure, the main new technical aspect 
is in showing that  these two structures satisfy suitable notions of {\it genericity} and leveraging on them to prove:
\begin{thm} \label{structure3}
The theory  of $(\qq; \SQ)$ is model complete, decidable, simple but not supersimple, and is $k$-independent for all $k \in \nn^{\geq 1}$.
\end{thm}

\noindent From above, $(\qq; \SQ)$ is ``less tame'' than $(\zz; \SF)$. The reader might therefore expect that $(\qq;<, \SQ)$ is wild. However, this is not the case:
\begin{thm} \label{structure4}
The theory $(\qq; <, \SQ)$ is model complete, decidable, is $\mathrm{NTP}_2$ but is not strong, and is $k$-independent for all $k \in \nn^{\geq 1}$.
\end{thm}

\noindent The paper is arranged as follows. In section 2, we define the appropriate notions of {\it genericity} for the structures under consideration. The model completeness and decidability results are proven in section 3 and the combinatorial tameness results are proven in section 4.

\medskip
\subsection*{Notation and conventions} Let $h, k$ and $l$ range over the set of integers and let  $m$, $n$, and $n'$ range over the set of natural numbers (which include zero). We let $p$ range over the set of prime numbers, and denote by $v_p$  the $p$-adic valuation on $\qq$.  Let $x$ be a single variable, $y$ a tuple of variables of unspecified length, $z$ the tuple $(z_1, \ldots, z_n)$ of variables, and $z'$ the tuple $(z'_1, \ldots, z'_{n'})$ of variables.  For an $n$-tuple $a$ of elements from a certain set, we let $a_i$ denote the $i$-th component of $a$ for $i \in \{1, \ldots, n\}$.  Suppose $G$ is an additive abelian group. We equip $G^m$ with a group structure by setting $+$ on $G^m$ to be the coordinate-wise addition. Viewing $G$ and $G^m$ as $\zz$-module, we define $ka$ with $a \in G$ and $kb$ with $b\in G^m$ accordingly. Suppose, $G$ is moreover an $L$-structure with $1_G$ the distinguished constant. We write $k$ for $k1_G$.  For $A \subseteq G$, we let $L(A)$ denote the language extending $L$ by adding constant symbols for elements of $A$ and view $G$ as an $L(A)$ structure in the obvious way.

\section{Genericity of the examples}
\noindent We study the structure $(\zz; \SF)$ indirectly by looking at its definable expansion to a richer language. For given $p$ and $l$, set 
$$U^\zz_{p, l} = \{ a \in \zz : v_p(a) \geq l\}.$$
Let $\sU^\zz = ( U^\zz_{p, l})$. The definition for $l \leq 0$ is not too useful as $U^\zz_{p, l} =\zz$ in this case. However, we still keep this for the sake of uniformity as we treat  $(\qq; \SQ)$ later. For $m>0$, set 
$$P^\zz_m = \{ a \in \zz : v_p(a) < 2 + v_p(m) \text{ for  all } p\}. $$ 
In particular,  $ P^\zz_1 = \SF$. Let $\sP^\zz = (P^\zz_m)_{m>0}$. We have that $(\zz, \sU^\zz, \sP^\zz)$ is a structure in the language $L_\mathrm{u}^*$ extending  $L_\mathrm{u}$ by families of unary predicate symbols for  $\sU^\zz$ and $( P^\zz_m)_{m>1}$. Note that $$U^\zz_{p,l}= \zz \text{ for } l\leq 0,\quad U^\zz_{p,l}= p^l\zz \text{ for } l>0, \text{ and} \quad   P^\zz_m =\bigcup_{d \mid m} d\SF \text{ for } m>0.$$ 
Hence, $U^\zz_{p, l}$ and $P^\zz_m$ are definable in   $( \zz, \SF)$, and so a subset of $\zz$ is definable in $(\zz; \sU^\zz, \sP^\zz)$ if and only if it is definable in $( \zz, \SF)$ .

\medskip \noindent
Let $(G; \sP^G, \sU^G)$ be an $L_\mathrm{u}^*$-structure. Then $\sU^G$ is a family indexed by pairs $(p, l)$, and $\sP^{G}$ is a family indexed by $m$. For $p$, $l$, and $m$, define $U^G_{p, l} \subseteq G$ to be  the member of $\sU^G$ with index $(p, l)$  and $P^G_m \subseteq G$ to be the member of the family $\sP^G$ with index $m$. In particular, we have
$\sU^G = (U^G_{p, l})$  and  $P^G = (P^G_m)_{m>0}$.
Clearly, this generalizes the previous definition for $\zz$.

\medskip \noindent We isolate the basic first-order properties of $(\zz; \sU^\zz, \sP^\zz)$. 
Let $\WSFZ$ be a recursive set of $L_\mathrm{u}^*$-sentences such that an $L_\mathrm{u}^*$-structure $(G; \sU^G, \sP^G)$ is a model of $\WSFZ$ if and only if $(G; \sU^G, \sP^G)$   satisfies the following properties:
\begin{enumerate}
\item[(Z1)] $(G;+,-,0,1)$ is elementarily equivalent to $(\zz;+,-,0,1)$; 
\item[(Z2)] $U^G_{p,l}= G$ for $l\leq 0$,  and $U^G_{p,l}= p^lG$ for $l>0$;
\item[(Z3)] $1$ is in $P^G_1$;
\item[(Z4)] for any given $p$, we have that $pa \in P^G_1$ if and only if $a\in P^G_1$ and $a \notin U^G_{p,1}$;
\item[(Z5)] $P^G_m =\bigcup_{d \mid m} dP^G_1 $ for all $m>0$.
\end{enumerate}
\noindent The fact that we could choose $\WSFZ$ to be recursive follows from the well-known decidability of $\zz$. Clearly, $(\zz; \sU^\zz, \sP^\zz)$ is a  model of $\WSFZ$. Several properties which hold in $(\zz; \sU^\zz, \sP^\zz)$ also hold in an arbitrary model of $\WSFZ$:

\begin{lem} \label{basicpropertiesZ}
Let $(G; \sU^G, \sP^G)$ be a model of $\WSFZ$. Then we have the following:
\begin{enumerate}
\item[$\mathrm{(i)}$] $(G; \sU^G)$ is elementarily equivalent to $(\zz; \sU^\zz)$;
\item[$\mathrm{(ii)}$] for all $k$, $p$, $l$, and $m>0$, we have  that
$$k \in U^G_{p,l} \text{ if and only if } k \in U^\zz_{p,l} \ \text{ and }\  k \in P^G_m \text{ if and only if } k \in P^\zz_m;$$
\item[$\mathrm{(iii)}$] for all $h\neq 0 $, $p$, and $l$, we have that $ha \in U^G_{p,l}$  if and only if $ a \in U^G_{p,l-v_p(h)}; $
\item[$\mathrm{(iv)}$] if $a\in G$ is in $U^G_{p, 2+v_p(m)}$ for some $p$, then $a \notin P^G_m$;
\item[$\mathrm{(v)}$]for all $h\neq 0$ and $m>0$, $ha \in P^G_{m}$ if and only if we have
$$ a \in P^G_{m} \ \text{ and }  \  a \notin U^G_{p,2+ v_p(m) -v_p(h)} \text{ for all } p \text{ which divides } h;$$
\item[$\mathrm{(vi)}$] for all $h>0$ and $m>0$, $a \in P^G_{m}$ if and only if  $ha \in P^G_{mh}$. 
\end{enumerate}
\end{lem}
\begin{proof}
Fix a model $(G; \sU^G, \sP^G)$ of $\WSFZ$.  It follows from (Z2) that the same first-order formula defines both  $U^G_{p,l}$ in $G$ and $U^\zz_{p,l}$ in $\zz$. Then using (Z1), we get $\mathrm{(i)}$. The first assertion of  $\mathrm{(ii)}$ is immediate from $\mathrm{(i)}$. Using this, (Z3), and (Z4), we get the second assertion of  $\mathrm{(ii)}$ for the case $m=1$. For $m \neq 1$,  we reduce to the case $m=1$ using property (Z5). Statement  $\mathrm{(iii)}$ is an immediate consequence of $\mathrm{(i)}$. We only prove below the cases $m = 1$ of $\mathrm{(iv-vi)} $ as the remaining cases of the corresponding statements can be reduced to these using (Z5). Statement $\mathrm{(iv)}$ is immediate for the case $m=1$ using (Z2) and (Z4). The case $m=1$ of $\mathrm{(v)}$ is precisely the statement of (Z4) when $h$ is prime, and then the proof proceeds by induction. For the case $m=1$ of $\mathrm{(vi)}$, ($\rightarrow$) follows from (Z5), and ($\leftarrow$) follows through a combination of Z5, $\mathrm{(v)}$ and induction on the number of prime divisors of $h$. 
\end{proof}

\noindent We next consider the structures $(\qq; \SQ)$ and  $(\qq; <,  \SQ)$. For given $p$, $l$, and $m>0$, in the same fashion as above,  we set 
$$U^\qq_{p, l} = \{ a \in \qq : v_p(a) \geq l\}\quad  \text{ and } \quad \ P^\qq_m = \{ a \in \qq : v_p(a) < 2 + v_p(m) \text{ for  all } p\},$$
 and let
$$\sU^\qq = ( U^\qq_{p, l})\quad  \text{ and } \quad  \sP^\qq = (P^\qq_m)_{m>0}.$$ 
Then $(\qq; \sU^\qq, \sP^\qq)$ is a structure in the language $L_\mathrm{u}^*$. Clearly, every subset of $\qq^n$ definable in $(\qq; \SQ)$ is also definable in $(\qq; \sU^\qq, \sP^\qq)$. A similar statement holds for  $(\qq; <, \SQ)$ and  $(\qq; <, \sU^\qq, \sP^\qq)$.  We will show that the reverse implications are also true.

\medskip \noindent 
The next lemma backs up the discussion on $\SQ$ and $\mathrm{ASF}^\qq$ preceding Theorem~\ref{structure3} in the introduction.
\begin{lem} \label{sumoftwosquarefree}
$\SF+\SF = \zz$, $\SQ+\SQ= \qq$, and $\mathrm{ASF}^\qq + \mathrm{ASF}^\qq=\{ a : v_p(a) >-2 \text{ for all } p\}$.
\end{lem}
\begin{proof}
We first prove that any integer $k$ is a sum of two elements from $\SF$. As $\SF = -\SF$ and the cases where $k=0$ or $k=1$ are immediate, we assume that $k>1$. It follows from \cite{Schrinelmanndensity} that the number of square-free positive integers less than $k$ is at least $\frac{53k}{88}$. Since $\frac{53}{88} > \frac{1}{2}$, this implies $k$ can be written as a sum of two positive square-free integers which gives us $\SF+\SF = \zz$. Using this, the other two equalities follow immediately.
\end{proof}

\begin{lem} \label{Definabilityofvaluation}
For all $p$ and $l$, $U^\qq_{p,l}$ is existentially $0$-definable in $( \qq; \SQ)$.
\end{lem}
\begin{proof}
As $U^\qq_{p, l+n} = p^nU^\qq_{p, l}$ for all $l$ and $n$, it suffices to show the statement for  $l=0$. Fix a prime $p$. We have for all $a \in \SQ$ that
$$  v_p(a) \geq 0  \  \text { if and only if } p^2 a \notin \SQ. $$ 
Using Lemma \ref{sumoftwosquarefree}, for all $a \in \qq$,  we have that $v_p(a) \geq 0$ if and only if there are $a_1, a_2 \in \qq$ such that   
$$ \big(a_1\in \SQ \wedge v_p(a_1) \geq 0\big) \wedge \big(a_2\in \SQ \wedge v_p(a_2) \geq 0\big) \ \text{ and  } \ a =a_1+a_2. $$
Hence, the set $U^\qq_{p, 0} = \{a \in \qq : v_p(a) \geq 0 \} $ is existentially definable in $(\qq; \SQ)$. The desired conclusion follows.
\end{proof}

\noindent It is also easy to see that for all $m$, $P_m^\qq = m \SQ$ for all $m>0$,  and so $P_m^\qq$ is existentially $0$-definable in $(\qq; \SQ)$. Combining with Lemma \ref{Definabilityofvaluation}, we get:

\begin{prop}\label{notintroducingnewdefinableset}
Every subset of $\qq^n$ definable in $(\qq; \sU^\qq, \sP^\qq)$ is also definable in $( \qq; \SQ)$. The corresponding statement for $(\qq; <, \sU^\qq, \sP^\qq)$ and $(\qq; <, \SQ)$ holds.
\end{prop}

\noindent In view of the first part of Proposition \ref{notintroducingnewdefinableset}, we can analyze $(\qq; \SQ)$ via $(\qq; \sU^\qq, \sP^\qq)$  in the same way we analyze $(\zz; \SF)$ via $(\zz; \sU^\zz, \sP^\zz)$. 
Let $\WSFQ$ be a recursive set of $L_\mathrm{u}^*$-sentences such that an $L_\mathrm{u}^*$-structure $(G; \sU^G, \sP^G)$ is a model of $\WSFQ$ if and only if $(G; \sU^G, \sP^G)$  satisfies the following properties:
\begin{enumerate}
\item[(Q1)] $(G;+,-,0,1)$ is elementarily equivalent to $(\qq;+,-,0,1)$; 
\item[(Q2)] for any given $p$, $U_{p,0}^G$ is an $n$-divisible subgroup of $G$ for all  $n$ coprime  with $p$;
\item[(Q3)] $1\in U_{p,0}^G$ and $1 \notin U_{p,1}^G$;
\item[(Q4)] for any given $p$, $p^{-l}U_{p,l}^G = U_{p, 0}^G$ if $l<0$  and $U_{p,l} = p^l U_{p, 0}$ if $l>0$;
\item [(Q5)]$ U_{p, 0}^G \slash U_{p, 1}^G$ is isomorphic as a group to $\zz \slash p\zz$; 
\item[(Q6)] $1 \in P_1^G$;
\item[(Q7)] for any given $p$, we have that $pa \in P_1^G$ if and only if $a\in P_1^G$ and $a \notin U_{p,1}^G$;
\item [(Q8)]$P_m^G = m P_1^G$ for $m>0$;
\end{enumerate}

\noindent The fact that we could choose $\WSFQ$ to be recursive follows from the well-known decidability of $\qq$. Obviously, $(\qq; \sU^\qq, \sP^\qq)$ is a  model of $\WSFQ$. 
Several properties which hold in $(\qq; \sU^\qq, \sP^\qq)$ also hold in an arbitrary model of $\WSFQ$:
 
\begin{lem} \label{basicpropertiesQ}
Let $(G; \sU^G, \sP^G)$ be a model of $\WSFQ$. Then we have the following:
\begin{enumerate}
\item[$\mathrm{(i)}$] For all $p$ and all  $l, l' \in \zz$ with $l\leq l'$, we have $U_{p,l}^G$ is a subgroup of $G$, $U_{p,l'}^G \subseteq U^G_{p,l}$. Further, we can interpret $U^G_{p, l} \slash U^G_{p, l'}$ as an $L$-structure with 1 being $p^l+U^G_{p, l'}$, and
$$U^G_{p, l} \slash U^G_{p, l'} \cong_{L} \zz\slash (p^{l'-l}\zz);$$
\item[$\mathrm{(ii)}$] for all $h$, $k\neq 0$, $p$, $l$, and $m>0$, we have  that
$$\frac{h}{k} \in U^G_{p,l} \text{ if and only if } \frac{h}{k} \in U^\qq_{p,l} \ \text{ and }\  \frac{h}{k} \in P^G_m \text{ if and only if } \frac{h}{k} \in P^\qq_m
$$
where $hk^{-1}$ is the obvious element in $\qq$ and in $G$;
\item[$\mathrm{(iii)}$] the replica of $\mathrm{(iii-vi)}$ of Lemma \ref{basicpropertiesZ} holds.
\end{enumerate}
\end{lem}
\begin{proof}
Fix a model $(G; \sU^G, \sP^G)$ of $\WSFQ$. From (Q2) we have that $U_{p,0}^G$ is a subgroup of $G$ for all $p$. It follows from (Q4) that $U^G_{p,l'} \subseteq U^G_{p,l} $ are subgroups of $G$ for all $p$ and $l\leq l'$.  With $U^G_{p, l} \slash U^G_{p, l'}$ being interpreted as an $L$-structure with 1 being $p^l+U^G_{p, l'}$, we get an $L$-embedding of $ \zz\slash (p^{l'-l}\zz)$ into $U^G_{p, l} \slash U^G_{p, l'}$ using (Q3) and (Q4). Further, we see that $|U^G_{p, l} \slash U^G_{p, l'}|=p^{(l'-l)}$ using (Q2)-(Q5) and induction on $l'-l$ together with the third isomorphism theorem; and so the aforementioned embedding must be an isomorphism, finishing the proof for $\mathrm{(i)}$. The first assertion of $\mathrm{(ii)}$ follows easily from (Q2)-Q(4). The second assertion for the case $m=1$ follows from the first assertion, (Q6), and (Q7). Finally, the case with $m\not=1$ follows from the case $m=1$ using (Q8). The proof for the replica of $\mathrm{(iii)}$ from Lemma \ref{basicpropertiesZ}  is a consequence of $\mathrm{(i)}$ and (Q4). The proofs for replicas of $\mathrm{(iv-vi)}$ from Lemma \ref{basicpropertiesZ} are similar to the proofs for $\mathrm{(iv-vi)}$ of Lemma \ref{basicpropertiesZ}.
\end{proof}

\noindent  As the reader may expect by now, we will study $(\qq; <, \SQ)$ via $(\qq; <, \sU^\qq, \sP^\qq)$. Let $L_\mathrm{ou}^*$ be $L_\mathrm{ou} \cup L_\mathrm{u}^*$.  Then $(\qq; <, \sU^\qq, \sP^\qq)$ can be construed as an  $L_\mathrm{ou}^*$-structure in the obvious way. Let $\WOSFQ$ be a recursive set of $L_\mathrm{ou}^*$-sentences such that an $L_\mathrm{ou}^*$-structure $(G; \sU^G, \sP^G)$ is a model of $\WOSFQ$ if and only if $(G; \sU^G, \sP^G)$   satisfies the following properties:
\begin{enumerate}
\item $(G; <)$ is elementarily equivalent to $(\qq;<)$;
\item $(G; \sU^G, \sP^G)$ is a model of $\WSFQ$.
\end{enumerate}
As $\text{Th}(\qq;<)$ is decidable, we could choose $\WOSFQ$ to be recursive.

\medskip\noindent Returning to the theory $\WSFZ$, we see that it does not fully capture all the first-order properties of $(\zz, \sU^\zz, \sP^\zz)$. For instance, we will show later in Corollary \ref{forinstance} that for all $c \in \zz$, there is $a \in \zz$ such that 
$$ a+c \in \SF\  \text{ and }\  a+c+1 \in \SF, $$
while the interested reader can construct models of $\WSFZ$ where the corresponding statement is not true. Likewise, the theories $\WSFQ$ and $\WOSFQ$ do not fully capture all the first-order properties of $(\qq; \sU^\qq, \sP^\qq)$ and $(\qq; <, \sU^\qq, \sP^\qq)$.

\medskip \noindent To give a precise formulation of the missing first-order properties of $(\zz, \sU^\zz, \sP^\zz)$, $(\qq; \sU^\qq, \sP^\qq)$, and $(\qq; < \sU^\qq, \sP^\qq)$, we need more terminologies. Let $t(z)$ be an $L_\mathrm{u}^*$-term (or equivalently an $L_\mathrm{ou}^*$-term) with variables in $z$.
An $L_\mathrm{u}^*$-formula (or an $L_\mathrm{ou}^*$-formula) which is a  boolean combination of formulas having the form $t(z) =0$ where we allow $t$ to vary is called an {\bf equational condition}. Similarly, an $L_\mathrm{ou}^*$-formula which is a boolean combination of formulas  having the form  $t(z) < 0$ where $t$ is allowed to vary is called an {\bf order-condition}. For any given $p$, $l$ define $t(z) \in U_{p,l}$ to be the obvious formula in $L_\mathrm{u}^*(z)$ which defines in an arbitrary $L_\mathrm{u}^*$-structure $(G; \sU^G, \sP^G)$ the set
$$ \{ c \in G^n : t^G(c) \in U^G_{p,l} \}.  $$
Define  the quantifier-free formulas $t(z) \notin U_{p,l}$, $t(z) \in P_m$,  and $t(z) \notin P_m$ in  $L_\mathrm{u}^*(z)$ for $p$, $l$, and for $m>0$ likewise.
For each prime $p$, an $L_\mathrm{u}^*$-formula (or an $L_\mathrm{ou}^*$-formula) which is a  boolean combination of formulas of the form $t(z) \notin U_{p,l}$ where $t$ and $l$ are allowed to vary is called a {\bf $p$-condition}. We call a $p$-condition as in the previous statement {\bf trivial} if the boolean combination is  the empty conjunction.

\medskip \noindent 
A {\bf parameter choice} of variable type $(x, z, z')$  is a triple $(k,m,\Theta)$ such that $k$ is in $\zz\setminus\{0 \}$, $m$ is in $\nn^{\geq 1}$, and  $\Theta = \big( \theta_p(x,z,z')\big)$ where $\theta_p(x, z, z')$ is a $p$-condition for each prime $p$ and is trivial for all but finitely many $p$.
We say that an $L_\mathrm{u}^*$-formula  $\psi(x, z, z')$ is {\bf special} if it has the form
$$    \bigwedge_{p} \theta_p(x, z, z') \wedge \bigwedge_{i=1}^n (kx+z_i\in P_{m}) \wedge \bigwedge_{i'=1}^{n'} (kx+z'_i \notin P_{m})   $$
where $k, m$ and $\theta_p(x, z, z')$ are taken from a parameter choice of variable type $(x, z, z')$. Every special formula  corresponds to a unique parameter choice and vice versa. Special formulas are   special enough that we have a ``local to global'' phenomenon in the structures of interest  but general enough to represent quantifier free formulas. We will explain the former point in the remaining part of the section and make the latter point precise with Theorem \ref{SimplifytoG}.

\medskip \noindent Let $\psi(x, z, z')$ be a special formula with parameter choice  $(k,m,\Theta)$ and $\theta_p(x, z, z')$ is the $p$-condition in $\Theta$ for each $p$. We define the {\bf associated equational condition} of  $\varphi(x, z, z' )$  to be the formula  
$$ \bigwedge_{i =1}^n  \bigwedge_{i'=1}^{n'} (z_i\neq z'_{i'})$$ 
and the {\bf associated $p$-condition} of  $\varphi(x, z, z' )$ to be the formula
$$ \theta_p(x, z , z') \wedge \bigwedge_{i=1}^n (kx+z_i\notin U_{p, 2+v_p(m)}). $$
It is easy to see that modulo $\WSFZ$ or $\WSFQ$, an arbitrary special formula implies its associated equational condition and its associated $p$-condition for any prime $p$.

\medskip \noindent  Suppose $(G; \sU^G, \sP^G)$ and $(H; \sU^H, \sP^H)$ are $L_\mathrm{u}^*$-structures such that the former is an $L_\mathrm{u}^*$-substructure of the latter. Let $\psi(x, z, z')$ be a special formula, $\psi_{=}(z,z')$  the associated equational condition, and $\psi_{p}(x,z,z')$ the associated $p$-condition for any given prime $p$.  For $c \in G^n$ and $c' \in G^{n'}$, we call the quantifier-free $L_\mathrm{u}^*(G)$-formula $\psi(x, c, c')$  a {\bf $G$-system}.  An element $a\in H$ such that $\psi(a, c, c')$ holds is called a {\bf solution} of $\psi(x, c, c')$ in $H$.
We say that $\psi(x, c, c')$ is {\bf satisfiable} in $H$ if it has a solution in $H$ and {\bf infinitely satisfiable} in $H$ if it has infinitely many solutions in $H$. 
 We say that $\psi(x, c, c')$ is {\bf nontrivial} if $ \psi_=(c, c')$ holds or more explicitly if $c$ and $c'$ have no common components.
For a given $p$, we say that $\psi(x, c, c')$ is {\bf $p$-satisfiable} in $H$ if there is $a_p\in H$  such that  $\psi_p(a_p, c, c')$ holds. A $G$-system  is {\bf locally satisfiable} in $H$ if it is  $p$-satisfiable in $H$ for all $p$.

\medskip \noindent  Suppose $(G; <, \sU^G, \sP^G)$ and $(H; <, \sU^H, \sP^H)$ are $L_\mathrm{ou}^*$-structures such that the former is an $L_\mathrm{ou}^*$-substructure of the latter.  All the definitions in the previous paragraph have obvious adaptations to this new setting as $(G; \sU^G, \sP^G)$ and  $(H;\sU^H, \sP^H)$ are $L_\mathrm{u}^*$-structures.
For $b$ and $b'$ in $H$ such that $b<b'$, define 
$$(b, b')^H = \{ a \in H: b<a<b'\}.$$ 
A $G$-system $\psi(x, c, c')$ is {\bf satisfiable in every $H$-interval} if it has a solution in the interval $(b, b')^H$ for all $b$ and $b'$ in $H$ such that $b<b'$. The following observation is immediate:

\begin{lem}\label{GlobaltoLocal}
Suppose $(G; \sU^G, \sP^G)$ is a model of either $\WSFZ$ or $\WSFQ$. Then every $G$-system which is satisfiable in $G$ is nontrivial and locally satisfiable in $G$. 
\end{lem}

\noindent It turns out that the converse and more are also true for the structures of interest. We say that a model $(G; \sU^G, \sP^G)$ of either $\WSFZ$ or  $\WSFQ$ is {\bf generic} if every nontrivial locally satisfiable $G$-system is infinitely satisfiable in $G$. An $\WOSFQ$ model $(G;<, \sU^G, \sP^G)$ is {\bf generic} if every nontrivial locally satisfiable $G$-system is satisfiable in every $G$-interval. We will later show that $( \zz; \sU^\zz, \sP^\zz)$,  $(\qq; \sU^\qq, \sP^\qq)$, and $(\qq; <, \sU^\qq, \sP^\qq)$ are generic.

\medskip\noindent Before that we will show that the above notions of {\it genericity} are first-order. Let $\psi(x, z, z')$ be the special formula corresponding to a parameter choice $(k,m,\Theta)$ with $\Theta = \big( \theta_p(x,z,z')\big)$. A {\bf boundary} of $\psi(x,z,z')$ is a number $B \in  \nn^{>0}$ such that $B>\max\{|k|,n\}$ and $\theta_p(x,z,z')$ is trivial for all $p>B$.
\begin{lem} \label{local condition 1}
Let $\psi (x, z, z')$ be a special formula, $B$ a boundary of $\psi(x,z,z')$, and $(G; \sU^G, \sP^G)$ a model of either $\WSFZ$ or $\WSFQ$. Then every $G$-system $\psi(x, c, c')$ is $p$-satisfiable for $p>B$.
\end{lem}
\begin{proof}
Let $\psi (x, z, z')$ be the special formula corresponding to a parameter choice $(k,m,\Theta)$, and  $B$, $(G; \sU^G, \sP^G)$ as in the statement of the lemma. Suppose $\psi(x, c, c')$ is a $G$-system, $p>B$, and $\psi_p(x,z,z')$ is the associated $p$-condition  of $\psi(x,z,z')$. Then $\psi_p(x, c, c')$ is equivalent to
$$  \bigwedge_{i=1}^n (kx+c_i \notin U_{p, 2+v_p(m)})\  \text{ in  } (G; \sU^G, \sP^G).$$
We will show a stronger statement that there is  a $a_p\in \zz$ satisfying the latter. Note that for all $d\notin U_{p,0}^G$, we have that $(ka+d \notin U_{p, 0})$ for all $a\in \zz$. From Lemma \ref{basicpropertiesQ}, we have that $U_{p,l}^G \subseteq U_{p,k}^G$ whenever $k<l$, so we can assume that $c_i \in U_{p,0}^G$ for $i \in \{1, \ldots, n\}$.  In light of  Lemma \ref{basicpropertiesZ} $\mathrm{(i)}$ and Lemma \ref{basicpropertiesQ} $\mathrm{(i)}$,  we have that 
$$ U^G_{p, 0}\slash U^G_{p, 2+v_p(m)} \cong_{L} \zz \slash (p^{2+v_p(m)}\zz).$$
It is easy to see that $k$ is invertible $\mathrm{mod}\ p^{2+v_p(m)}$ and that $p^{2+v_p(m)}>n$. Choose $a_p$ in $\{0, \ldots, p^{2+v_{p}(m)}-1\}$  such that the images of $ka_p+c_1, \ldots, ka_p+c_n$ in $\zz \slash (p^{2+v_p(m)}\zz)$ are not $0$. We check that $a_p$ is as desired.
\end{proof}

\begin{cor}
There is an $L_\mathrm{u}^*$-theory $\TSFZ$ such that the models of $\TSFZ$ are the  generic models of $\WSFZ$. Similarly, there is an $L_\mathrm{u}^*$-theory $\TSFQ$ and an $L_\mathrm{ou}^*$-theory $\TOSFQ$ satisfying the corresponding condition for $\WSFQ$ and $\WOSFQ$.
\end{cor}

\noindent In the rest of the paper, we fix $\TSFZ$, $\TSFQ$, and $\TOSFQ$ to be as in the previous lemma.  We can moreover arrange them to be recursive. In the remaining part of this section, we will show that  $(\zz; \sU^\zz, \sP^\zz)$, $(\qq; \sU^\qq, \sP^\qq)$ and $(\qq; <, \sU^\qq, \sP^\zz)$ are models of  $\TSFZ$, $\TSFQ$, and $\TOSFQ$ respectively. The proof that the latter are in fact the full axiomatizations of the theories of the former needs to wait until next section. Further we fix $\SFZ$ and $\SFQ$ to be the theories whose models are precisely the $L_{\mathrm{u}}$-reducts of models of $\TSFZ$ and $\TSFQ$ respectively, and $\OSFQ$ to be the theory whose models are precisely $L_{\mathrm{ou}}$ reducts of models of $\TOSFQ$.  For the reader's reference, the following table lists all the languages, the corresponding theories and primary structures under consideration:

\vspace{0.1in}
\begin{center}
\begingroup
\renewcommand{\arraystretch}{1.25}
\begin{tabular}{|c|c|c|} 
\hline
Languages & Theories & Primary structures  \\ [0.2ex] 
\hline\hline 
$L$ & Th($\zz$), Th($\qq$) & $\zz, \qq$\\ [0.2 ex]
\hline
$L_{\mathrm{u}}$  & $\SFZ$, $\SFQ$ & $(\zz; \SF)$, $(\qq; \SQ)$  \\ [0.2ex]  
\hline
$L_{\mathrm{ou}}$  &$\OSFQ$ & $(\zz; <, \SF)$, $(\qq;<, \SQ)$  \\ [0.2ex] 
\hline
$L_{\mathrm{u}}^*$ & $\WSFZ$, $\TSFZ$, $\WSFQ$, $\TSFQ$& $(\zz; \sU^\zz, \sP^\zz)$, $(\qq; \sU^\qq, \sP^\qq)$\\ [0.2ex] 
\hline
$L_{\mathrm{ou}}^*$ & $\WOSFQ$, $\TOSFQ$ & $(\qq; <, \sU^\qq, \sP^\qq)$  \\ [0.2ex] 
\hline
\end{tabular}
\endgroup
\end{center}

\medskip \noindent Suppose $h\neq 0$. For a term $t(z) = k_1 z_1+ \ldots +k_n z_n +e$, let $t^h(z)$ be the term $k_1z_1+ \ldots+k_nz_n+ he$.  If $\varphi(z)$ is  a boolean combination of atomic formulas of the form $t(z) \in U_{p,l}$ or $t(z) \in P_{m}$ where $t(z)$ is an $L_\mathrm{u}^*$-term, define $\varphi^h(z)$ to be the formula obtained by replacing $t(z) \in U_{p,l}$ and $t(z) \in P_{m}$ in $\varphi(z)$ with $t^h(z) \in U_{p,l+v_p(h)}$ and $t^h(z) \in P_{mh}$ for every choice of  $p$, $l$, $m$ and $L_\mathrm{u}^*$-term $t$. It follows from Lemma \ref{basicpropertiesZ} $\mathrm{(iii)}$, $\mathrm{(vi)}$ and Lemma \ref{basicpropertiesQ} $\mathrm{(iii)}$ that across models of $\WSFZ$ and $\WSFQ$,
 $$\varphi^h(hz) \text{ is equivalent to } \varphi(z). $$
Moreover, if $\theta(z)$ is a $p$-condition, then  $\theta^h(z)$ is also  $p$-condition. If $\psi(x,z, z')$ is the special formula corresponding to a parameter choice $(k, m, \Theta)$ with $\Theta = \big(\theta_p(x,z,z')\big)$, then $\psi^h(x,z, z')$ is the special formula  corresponding to the parameter choice  $(k, hm, \Theta^h)$ with $\Theta^h = \big(\theta_p^h(x,z,z')\big)$. It is easy to see from here that:

\begin{lem}\label{BoundaryRobust}
For $h \neq 0$, any boundary of a special formula $\psi(x,z,z')$ is also a boundary of $\psi^h(x,z,z')$ and vice versa.
\end{lem}
\noindent Let $\psi(x, z, z')$ be a special formula, $(G; \sU^G,\sP^G)$ a model of either $\WSFZ$ or $\WSFQ$, and $\psi(x, c, c')$ a $G$-system. Then $\psi^h(x, hc, hc')$ is also a $G$-system which we refer to as the {\bf $h$-conjugate} of $\psi(x, c, c')$. This has the property that $\psi^h(ha,hc,hc')$ if and only if $\psi(a,c,c')$ for all $a \in G$.

\medskip \noindent For $a$ and $b$ in $\zz$, we write $a\equiv_{n}b$ if $a$ and $b$ have the same remainder when divided by $n$. We need the following version of Chinese remainder theorem:
\begin{lem}\label{condense}
Suppose $B$ is in $\nn^{>0}$, $\Theta$ is a family $\big(\theta_p(x,z)\big)_{p \leq B}$ of $L^*_{\mathrm{u}}$-formulas with $\theta_p (x,z)$ being a $p$-condition for each $p\leq B$, and $c \in \zz^n$ is such that $\theta_p(x,c)$ defines a nonempty set in $(\zz; \sU^\zz, \sP^\zz)$ for all $p \leq B$. Then we can find $D$ $\in\nn^{>0}$ and $r \in  \{0, \ldots, D-1\} $ such that for all  $h\neq 0$ with  $\mathrm{gcd}(h,B!)=1$, we have 
$$a\equiv_{D} hr \text{ implies } \bigwedge_{p \leq B} \theta^h_p(a,hc)\ \  \text{ for all } a \in \zz.$$ 
\end{lem}
\begin{proof}
Let $B$, $\Theta$, and $c$ be as stated. Fix $h\neq 0$ such that $\mathrm{gcd}(h,B!)=1$. Hence, $v_p(h)=0$ for $p\leq B$, and so  the $p$-condition $\theta_p^h(x,z)$ is obtained from the $p$-condition $\theta_p(x,z)$ by replacing any atomic formula  $kx +t(z) \in U_{p, l}$ appearing in $\theta_p(x,z)$  with $kx +t^h(z) \in U_{p, l}$.
 Now for $p\leq B$, let $l_p$ be the largest value of $l$ occurring in an atomic formula in  $\theta_p(x,z)$. 
Set 
$$D = \prod_{p \leq B} p^{l_p}. $$
Obtain $a_p \in \zz$ such that $\theta_p(a_p, c)$ holds in $(\zz; \sU^\zz, \sP^\zz)$. Equivalently, we have $\theta^h_p(ha_p, hc)$ holds in $(\zz; \sU^\zz, \sP^\zz)$.  By the Chinese remainder theorem, we get $r$ in $\{0, \ldots, D-1\}$ such that 
$$r\equiv_{p^{l_p}} a_p\quad \text{ for all } p \leq B.$$ 
We check that $r$ is as desired. Suppose $a \in \zz$ is such that $a \equiv_D hr$. By construction, if $p\leq B$, $l \leq l_p$, and 
$kx+t(z) \in U_{p,l }$ is any atomic formula, then  $ka +t^h(hc) \in U^\zz_{p, l}$ if and only if $k(ha_p) +t^h(hc) \in U^\zz_{p, l}$. It follows that $\theta^h_p( a, hc)$ is equivalent to $\theta^h_p( ha_p, hc)$ in $(\zz; \sU^\zz, \sP^\zz)$. Thus $\theta^h_p( a, hc)$ holds for all $p \leq B$.  
\end{proof}

\noindent Towards showing that the structures of interest are generic, the key number-theoretic ingredient we need is the following result:

\begin{lem} \label{positivedensity} Let $\psi(x, z, z')$ be a special formula and $\psi(x,c,c')$ a nontrivial $\zz$-system which is locally satisfiable in $\zz$. For $h > 0$, and $s, t \in \qq$ with $s< t$, set 
$$\Psi^h(hs,ht) = \{ a \in \zz: \psi^h(a,hc,hc') \text{ holds and } hs<a<ht\}.$$ Then there exists $N \in \nn^{>0}$,  $\varepsilon\in (0,1)$, and $C\in \rr$ such that for all $h > 0$ with $\mathrm{gcd}(h,N!)=1$ and $s, t \in \qq$ with $s<t$, we have that
$$|\Psi^h(hs,ht)| \geq \varepsilon h(t-s) - \left(\sum_{i=1}^{n}\sqrt{|hks+hc_i|} + \sqrt{|hkt+hc_i|}\right)  + C. $$ 
\end{lem}
\begin{proof}
Throughout this proof, let  $\psi(x,z,z')$, $\psi(x,c,c')$,  and $\Psi^h(hs,ht)$ be as stated. We first make a number of observations. Suppose $\psi(x,z,z')$ corresponds to the parameter choice $(k,m,\Theta)$ and has a boundary $B$, and $\psi_p(x,z,z')$ is the associated $p$-condition of  $\psi(x,z,z')$.  Then $\psi^h(x,z,z')$ corresponds to the parameter choice $(k,hm,\Theta^h)$, and $B$ is also a boundary of $\psi^h(x,z,z')$ by Lemma \ref{BoundaryRobust}. Moreover $\psi^h_p(x,z,z')$ is the associated $p$-condition of $\psi^h(x,z,z')$. Since $\psi(x,c,c')$ is locally satisfiable in $\zz$, we can use Lemma \ref{condense} to fix $D\in\nn^{>0}$ and $r\in \{0,\ldots, D-1\}$ such that for each $h > 0$ with $\mathrm{gcd}(h,B!)=1$, we have
$$  a\equiv_{D} hr \text{ implies } \bigwedge_{p \leq B} \psi^h_p(a,hc, hc')\ \  \text{ for all } a \in \zz.  $$
We note that $D$ here is independent of the choice of $h$ for all $h$ with $\mathrm{gcd}(h,B!)=1$.

We introduce a variant of  $\Psi^h(hs,ht)$ which is needed in our estimation of $|\Psi^h(hs,ht)|$. Until the end of the proof, set $l_p =2 + v_p(m)$.  Fix primes $p_1, \ldots, p_{n'}$ such that $p_1 > c_i$ for all  $i \in \{1, \ldots, n\}$, $p_1 > c'_{i'}$ for all $i' \in \{1, \ldots, n'\}$ and $$ B<p_1 < \ldots < p_{n'}.$$  For $M>p_{n'}$, $h > 0$ with $\mathrm{gcd}(h,B!)=1$, define $\Psi^h_{M}(hs,ht)$ to be the set of $a \in \zz$ such that $hs<a<ht$ and
$$(a\equiv_{D} hr) \wedge\bigwedge_{B< p \leq M}\big( \bigwedge_{i=1}^{n}(ka+hc_i \not \equiv_{p^{l_{p}+v_p(h)}} 0)\big)\wedge\bigwedge_{i'=1}^{n'}(ka+hc'_{i'} \notin P_{hm}^{\zz}).$$
It is not hard to see that $\Psi^h(hs,ht)\cap \{a\in \zz: a\equiv_D hr\}\subseteq \Psi^h_{M}(hs,ht)$, and the latter is intended to be an upper approximation of the former. The desired lower bound for $|\Psi^h(hs,ht)|$ will be obtained via a lower bound for $|\Psi^h_{M}(hs,ht)|$ and an upper bound for $|\Psi^h_{M}(hs,ht)\setminus \Psi^h(hs,ht)| $. 

Now we work towards establishing a lower bound on $|\Psi^h_{M}(hs,ht))|$ in the case where $M> p_{n'} $, $h > 0$, and $\text{gcd}(h, M!)=1$. The latter assumption implies in particular that $p^{l_p + v_p(h)} = p^{l_p}$ for all $p \leq M$. For $p>B $, we have that $p>|k|$ and so $k$ is invertible $\mathrm{mod}\ p^{l_p}$.
Set $$\Delta = \{p :  B < p \leq M\}\setminus \{ p_{i'} : 1 \leq i' \leq n'  \}.$$ 
For  $p\in \Delta$, as $k$ is invertible $\mathrm{mod}\ p^{l_p}$, there are at least $p^{l_{p}}- n$ (note we have $p>B>n$) choices of $r_p$  in $\{0, \ldots, p^{l_{p}}-1 \}$ such that  if $ a \equiv_{p^{l_p}} r_{p}$, then
$$ \bigwedge_{i=1}^n  (ka+hc_i \not \equiv_{p^{l_p}} 0). $$ 
Suppose $p = p_{i'}$ for some $i' \in \{1, \ldots, n'\}$. By the assumption that $\psi(x,c,c')$ is nontrivial, $c$ has no common components with $c'$. Since $\text{gcd}(h, M!)=1$,  $h$ and $p$ are coprime, and so the components of $hc$ and $hc'$ are pairwise distinct $\mathrm{mod}\ p^{l_p}$. As $k$ is invertible $\mathrm{mod}\ p^{l_p}$, there is exactly one  $r_p$ in  $\{ 0, \ldots, p^{l_{p}}-1 \}$ such that if $a \equiv_{p^{l_p}} r_p$, then
$$  \bigwedge_{i=1}^n (ka+hc_i \not \equiv_{p^{l_p}} 0) \wedge (ka+hc'_{i'} \equiv_{p^{l_p}} 0)  \ \text{ and consequently }\  ka+hc'_{i'} \notin P_{hm}^{\zz}  .$$
Now it follows by the Chinese remainder theorem that,
$$|\Psi^h_{M}(hs,ht)|\geq \left\lfloor \frac{ht-hs}{D\prod_{B<p\leq M} p^{l_p}} \right\rfloor \prod_{p \in \Delta}\left(p^{l_p}-n\right).$$
Then it follows that,
$$|\Psi^h_{M}(hs,ht)|\geq \frac{ht-hs}{D}\prod_{p \leq p_{n'}}\frac{1}{p^{l_p}}\prod_{p> p_{n'}}^{\leq M} \left(1-\frac{n}{p^{l_p}}\right) - \prod_{p\leq M} p^{l_p}. $$
Set
$$\varepsilon = \frac{1}{2D}\prod_{p \leq p_{n'}}\frac{1}{p^{l_p}}\prod_{p > p_{n'}} \left(1-\frac{n}{p^{l_p}}\right).$$
Now as $l_p \geq 2$, for $U$ $\in \nn^{>0}$ with $U>\max\{p_{n}',n^2\}$ we have that
$$\prod_{p > U} \left(1-\frac{n}{p^{l_p}}\right)>\prod_{p>U}\left(1-\frac{1}{p^{\frac{3}{2}}}\right).$$
Hence, it follows from  Euler's product formula that $\varepsilon>0$. We now have 
$$|\Psi^h_{M}(hs,ht)| \geq 2\varepsilon(ht-hs) - \prod_{p\leq M} p^{l_p}.$$ 
We note that $\varepsilon$ is independent of the choice of $M$ and $h$, and will serve as the promised $\varepsilon$ in the statement of the lemma.

Next we obtain a upper bound on $|\Psi^h_{M}(s,t)\setminus \Psi^h(s,t)|$ for $M> p_{n'} $ $h > 0$ and $\text{gcd}(h, M!)=1$. We  arrange that $k>0$ by replacing $c$ by $-c$ and $c'$ by $-c'$ if necessary. Note that an element   $a \in \Psi^h_{M}(s,t) \setminus  \Psi^h(s,t)$ must be such that
$$hks+hc_i < ka+hc_i < hkt+hc_i \quad \text{for all } i \in \{1, \ldots, n\}$$ and  $ka+hc_i$ is a multiple of $p^{l_p}$ for some $p>M$ and $i \in \{1, \ldots, n\}$. For each $p$ and $i \in \{1, \ldots, n\}$, the number of non-zero multiples of $p^{l_p}$ in $(hks+hc_i, hkt+hc_i)$ is 
$$\lfloor hk(t -s) p^{-l_p}\rfloor -2, \ \text{or}\  \lfloor hk(t -s) p^{-l_p}\rfloor -1, \ \text{or}\ \lfloor hk(t -s) p^{-l_p}\rfloor, \ \text{or}\  \lfloor hk(t -s) p^{-l_p} \rfloor +1.$$
In the last case, as $l_p \geq 2$ we  moreover have
$$ p^2\leq |hks+hc_i| \quad \text{ or }\quad p^2\leq |hkt+hc_i|,$$
and so
$$p\leq \sqrt{|hks+hc_i|} + \sqrt{|hkt+hc_i|}.$$ 
As $l_p \geq 2$, we have  $ \lfloor hk(t -s) p^{-l_p}\rfloor \leq hk(t-s)p^{-2}$. Therefore we have that 
$$|\Psi^h_{M}(s,t)\setminus \Psi^h(s,t)|\leq h(t-s)\sum_{p>M}\frac{nk}{p^2} + \left(\sum_{i=1}^{n}\sqrt{|hks+hc_i|} + \sqrt{|hkt+hc_i|}\right) +1 .$$ 

We now obtain $N$ and $C$ as in the statement of the lemma. Note that $$\sum_{p>T}p^{-2} \leq \sum_{n>T}n^{-2} = O( T^{-1}). $$ Using this, we obtain $N \in \nn^{>0}$ such that $N>p_{n'}$ and $ \sum_{p>N}knp^{-2}< \varepsilon$ where $\varepsilon$ is from the preceding paragraph. Set $C = - \prod_{p\leq N} p^{l_p} -1 $. Combining the estimations from the preceding two paragraphs for $M=N$ it is easy to see that $\varepsilon, N, C$ are as desired.
\end{proof}

\begin{rem}\label{forinstance}
The above weak lower bound is all we need for our purpose. We expect that a stronger estimate can be obtained using modifications of available techniques in the literature; see for example \cite{Mirsky}.
\end{rem}

\begin{cor}
For all $c\in \zz$, there is $a\in \zz$ such that 
$$ a+c \in \SF\  \text{ and }\  a+c+1 \in \SF. $$
\end{cor}
\begin{proof}
We have that for all $c\in \zz$, $\psi(x,c)=(x+c\in \SF)\wedge(x+c+1\in \SF)$ is a locally satisfiable $\zz$-system. Applying Lemma \ref{positivedensity} for $h=1$, $s=0$, and $t$ sufficiently large we see there is a solution $a\in\zz$ for $\psi(x,c)$.
\end{proof}

\noindent We next prove the main theorem of the section:

\begin{thm} \label{genericity}
The $\WSFZ$-model $( \zz; \sU^\zz, \sP^\zz)$, the $\WSFQ$-model $(\qq; \sU^\qq, \sP^\qq)$,  and the $\WOSFQ$-model $(\qq; <, \sU^\qq, \sP^\qq)$  are generic.
\end{thm}
\begin{proof}
We get the first part of the theorem by applying Lemma \ref{positivedensity} for $h =1$, $s =0$, and $t$ sufficiently large.  As the second part of the theorem follows easily from the third part, it will be enough to show that the $\WOSFQ$-model $(\qq; <, \sU^\qq, \sP^\qq)$ is generic.  Throughout this proof, suppose $\psi(x,z,z)$ is a special formula and  $\psi(x, c, c')$  is  a $\qq$-system which is nontrivial and locally satisfiable in $\qq$.  Our job is to show that the $\qq$-system $\psi(x, c, c')$ has a solution in the $\qq$-interval $(b, b')^\qq$ for an arbitrary choice of $b, b' \in \qq$ such that $b<b'$.

We first reduce to the special case where $\psi(x, c, c')$ is also a $\zz$-system which is nontrivial and locally satisfiable in $\zz$. Let $B$ be the boundary of $\psi(x,z,z')$ and for each $p$, let $\psi_p(x,z,z')$ be the associated $p$-condition of $\psi(x,z,z')$. Using the assumption that $\psi(x, c, c')$ is locally satisfiable $\qq$-system, for each $p<B$ we obtain $a_p \in \qq$ such that $\psi_p(a_p,c,c')$ holds. Let $h>0$ be such that  
$$ hc \in \zz^n, hc' \in \zz^{n'} \text{ and } ha_p \in \zz \text{ for all } p<B. $$
Then by the choice of $h$ , Lemma \ref{local condition 1}, and Lemma \ref{BoundaryRobust}, the $h$-conjugate $\psi^h(x,hc,hc')$ of $\psi(x, c, c')$ is a $\zz$-system which is nontrivial and locally satisfiable in $\zz$.
On the other hand, $\psi(x,c,c')$ has a solution in a interval $(b,b')^\qq$ if and only if $$\psi^h(x,hc,hc') \text{ has a solution in } (hb,hb')^{\qq}.$$ Hence,  by replacing $\psi(x,z,z')$ with $\psi^h(x,z,z')$, $\psi(x, c, c')$ with $\psi^h(x,hc,hc')$, and $(b, b')^{\qq}$ with $(hb, hb')^{\qq}$ if necessary we get the desired reduction. 

We show $\psi(x, c, c')$ has a solution in the $\qq$-interval $(b, b')^\qq$ for the special case in the preceding paragraph. By an argument similar to the preceding paragraph, it suffices to show that for some $h\neq 0$, $\psi^h(x, hc, hc')$ has a solution in $(hb, hb')^\qq$. Applying Lemma \ref{positivedensity} for $s =b$, $t=b'$, and $h$ sufficiently large satisfying the condition of the lemma, we get the desired conclusion.
\end{proof}

\section{Logical Tameness}
\noindent We will next prove that $\TSFZ$, $\TSFQ$, and $\TOSFQ$ admit quantifier elimination. We first need a technical lemma saying that modulo  $\WSFZ$ or $\WSFQ$, an arbitrary quantifier free formula $\phi(x,y)$ is not much more complicated than a special formula; recall that $x$ always denotes a single variable.

\begin{lem}\label{SimplifytoG}
Suppose $\varphi(x,y)$ is a quantifier-free $L^*_{\mathrm{u}}$-formula. Then $\varphi(x,y)$ is  equivalent modulo $\WSFZ$ to a disjunction of quantifier-free formulas of the form
$$  \rho(y) \wedge \varepsilon(x, y)  \wedge \psi(x, t(y), t'(y))  $$
where
\begin{enumerate}
\item[$\mathrm{(i)}$] $t(y)$ and $t'(y)$ are tuples of $L^*_{\mathrm{u}}$-terms with length $n$ and $n'$respectively;
\item[$\mathrm{(ii)}$] $\rho(y)$ is a quantifier-free $L^*_{\mathrm{u}}$-formula, $ \varepsilon(x, y)$ an equational condition, $\psi(x, z, z')$ a special formula.
\end{enumerate}
The corresponding statement with $\WSFZ$ replaced by $\WSFQ$ also holds.
\end{lem}
\begin{proof}

Let  $\varphi(x, y)$ be a quantifier-free $L^*_{\mathrm{u}}$-formula. We will use the following disjunction observation several times in our proof: If $\varphi(x,y)$ is a finite disjunction of quantifier-free $L^*_{\mathrm{u}}$-formulas and we have proven the desired statement for each of those, then the desired statement for $\varphi(x,y)$ follows. In particular, it allows us to assume that $\varphi(x, y)$ is the conjunction
$$ \rho(y)  \wedge \varepsilon(x, y) \wedge \bigwedge_p \eta_p(x, y) \wedge \bigwedge_{i=1}^n (k_ix+t_i(y) \in P_{m_i}) \wedge \bigwedge_{i=1}^{n'} (k'_ix+t'_{i}(y) \notin P_{m'_i})  $$
where $\rho(y)$ is a quantifier-free $L^*_{\mathrm{u}}$-formula, $\varepsilon(x, y)$ is an equational condition,   $k_1, \ldots, k_n$ and $k_1', \ldots, k'_{n'}$ are in $\zz\setminus\{0\}$, $m_1, \ldots, m_n$ and $m'_1, \ldots, m'_{n'}$ are in $\nn^{\geq 1}$, $t_1(y), \ldots, t_n(y)$ and $t'_1(y), \ldots, t'_n(y)$ are $L^*_{\mathrm{u}}$-terms with variables in $y$,  $\eta_p(x,y)$ is a $p$-condition for each $p$, and $\eta_p(x,y)$ is trivial for all but finitely many $p$. 

We make further reductions to the form of $\varphi(x,y)$. Set $t(y) =(t_1(y), \ldots, t_n(y))$ and $t'(y)=(t'_1(y), \ldots, t'_{n'}(y))$.  Using the disjunction observation and the fact that 
$$(x+ y_j \in P_1) \vee (x + y_j \notin P_1)$$ is a tautology for every component $y_j$ of $y$, we can assume that either $x+ y_j \in P_1$ or $x + y_j \notin P_1$ are among the conjuncts of $\varphi(x, y)$, and so $y_j$ is among the components of $t(y)$ or $t'(y)$. Then we obtain for each prime $p$ a $p$-condition $\theta_p(x, z, z')$ such that $\theta_p(x, t(y), t'(y))$ is logically equivalent to $\eta_p( x, y)$. Let $\xi(x, z, z')$ be the formula 
$$ \bigwedge_p \theta_p(x, z, z') \wedge \bigwedge_{i=1}^n (k_ix+z_i\in P_{m_i}) \wedge \bigwedge_{i=1}^{n'} (k'_ix+z'_i \notin P_{m'_i}).$$
Clearly, $\varphi(x,y)$ is equivalent to the formula $\rho(y) \wedge \varepsilon(x, y) \wedge \xi(x, t(y), t'(y))$, so we can assume that $\varphi(x,y)$ is the latter.

We need a small observation. For a $p$-condition $\theta_p(z)$ and $h \neq 0$, we will show that there is another $p$-condition $\eta_p(z)$ such that modulo $\WSFZ$ and $\WSFQ$, 
$$ \eta_p(z_1, \ldots, z_{i-1}, hz_i, z_{i+1}, \ldots, z_n) \  \text{ is equivalent to } \  \theta_p(z). $$
For the special case where $\theta_p(z)$ is $t(z) \in U_{p, l}$, the conclusion follows from Lemma \ref{basicpropertiesZ}(iii), Lemma \ref{basicpropertiesQ}(iii) and the fact that there is an $L^*_{\mathrm{u}}$-term $t'(z)$ such that $t'(z, \ldots, z_{i-1}, hz_i, z_{i+1}, \ldots, z_n) = ht(z).$
The statement of the paragraph follows easily from this special case. 

With $\varphi(x,y)$ as in the end of the second paragraph, we further reduce the main statement to the special case where there is $k\neq 0$ such that $k_i =  k'_{i'} =k$ for all $i\in \{1, \ldots, n\}$ and $i' \in \{ 1, \ldots, n'\}$. Choose $k\neq 0 $ to be a common multiple of $k_1, \ldots, k_n$ and $k'_1, \ldots k'_{n'}$. Then by Lemma \ref{basicpropertiesZ}(vi) and Lemma \ref{basicpropertiesQ}(iii), we have for each $i \in \{1, \ldots, n\}  $ that
$$ k_ix+z_i \in P_{m_i}\  \text{ is equivalent to }\ (kx+kk^{-1}_iz_i \in P_{kk^{-1}_im_i})   \text{ modulo either } \WSFZ \text{ or }\WSFQ. $$
We have a similar observation for $k$ and $k'_{i'}$ with $i' \in \{1, \ldots, n'\}$. The desired reduction easily follows from these observations and the preceding paragraph.

Continuing with the reduction in the preceding paragraph, we next arrange that  there is $m > 0$ such that $m_i =  m'_{i'} =m$ for all $i\in \{1, \ldots, n\}$ and $i' \in \{ 1, \ldots, n'\}$.  Let $m$ be a common multiple of $m_1, \ldots, m_n$ and $m'_1, \ldots m'_{n'}$.  By Lemma \ref{basicpropertiesZ}(v, vi) and Lemma \ref{basicpropertiesQ}(iii), we have for $i \in \{ 1, \ldots, n\}$ that modulo either $\WSFZ$  or $\WSFQ$
$$ kx+z_i \in P_{m_i} \  \text{ is equivalent to  }\ kx+z_i \in P_{m} \wedge \bigwedge_{p \mid \frac{m}{m_i}}  kx+z_i \notin U_{p, 2+ v_p(m_i)} $$
and for $i' \in \{ 1, \ldots, n'\}$ that modulo either $\WSFZ$  or $\WSFQ$
$$ kx+z'_{i'} \notin P_{m'_{i'}} \  \text{ is equivalent to }\ kx+z'_{i'} \notin P_{m} \vee \bigvee_{p \mid \frac{m}{m'_{i'}}} kx+z'_{i'} \in U_{p,2+ v_p(m'_{i'})} .  $$
It follows that $\varphi(x, y)$ is equivalent to a disjunction of formulas of the form we are aiming for. The desired conclusion of the lemma follows from the  disjunction observation.
\end{proof}

\begin{cor}\label{SimplifytoG2}
Suppose $\varphi(x,y)$ is a quantifier-free $L^*_{\mathrm{ou}}$ formula. Then $\varphi(x,y)$ is  equivalent modulo $\WOSFQ$ to a disjunction of quantifier-free formulas of the form
$$  \rho(y) 
\wedge \lambda(x,y) \wedge \psi(x, t(y), t'(y))  $$
where
\begin{enumerate}
\item[$\mathrm{(i)}$] $t(y)$ and $t'(y)$ are tuples of $L^*_{\mathrm{ou}}$-terms with length $n$ and $n'$respectively;
\item[$\mathrm{(ii)}$] $\rho(y)$ is a quantifier-free $L^*_{\mathrm{ou}}$-formula, $ \lambda(x, y)$ an order condition, $\psi(x, z, z')$ a special formula.
\end{enumerate}
\end{cor}

\noindent In the next lemma, we show a ``local quantifier elimination'' result.
\begin{lem} \label{localQE}
If $\varphi(x,z)$ is a $p$-condition, then modulo either $\WSFZ$ or $\WSFQ$, the formula  $\exists x \varphi(x,z)$ is equivalent to a $p$-condition $\psi(z)$.
\end{lem}
\begin{proof}
If $\varphi(x,z)$ is a $p$-condition, then by Lemma \ref{basicpropertiesZ} $(\mathrm{i})$, modulo $\WSFZ$, it is a boolean combination of atomic formulas of the form $kx +t(z) \in U_{p, l}$ where $t(z)$ is an $L^*_{\mathrm{u}}$-term, and $l>0$. Let $l_p$ be the largest value of $l$ occurring in such atomic formulas, and set $$S=\{(m_1,\ldots,m_n): 0 \leq m_i < p^{l_p} \text{ for each $i$, and } (\zz; \sU^\zz) \models \exists x \varphi(x,m_1,\ldots,m_n)\}.$$ Then by Lemma \ref{basicpropertiesZ} $(\mathrm{i})$, modulo $\WSFZ$, $\exists x \varphi(x,z)$ is equivalent to the $p$-condition $\bigvee_{(m_1,\ldots,m_n)\in S} (\bigwedge_{i=1}^{n} (z_i\equiv_{p^{l_p}} m_i))$ .

Now, we proceed to prove the statement for models of $\WSFQ$. Throughout the rest of the  proof, suppose $\varphi(x,z)$ is a $p$-condition, $k$, $k'$, $l$, $l'$ are in $\zz$, and $t(z), t'(z)$ are $L_{\mathrm{u}}^*$-terms. First, we consider the case where $\varphi(x,z)$ is a  $p$-condition of the form 
$kx+t(z) \in U_{p,l}$. 
The case $k=0$ is trivial.  If $k \neq 0$, then $\exists x (kx+t(z) \in U_{p,l})$ is tautological modulo $\WSFQ$ following from (Q1) in the definition of $\WSFQ$ and Lemma \ref{basicpropertiesQ}(i).

We next consider the case where $\varphi(x,z)$ is a finite  conjunction of $p$-conditions in $L^*_{\mathrm{u}}(x, z)$ such that one of the conjuncts is $kx+t(z) \in U_{p,l}$ with $k \neq 0$ and the other conjuncts are either of the form $k'x+t'(z) \in U_{p,l'}$ or of the form $k'x+t'(z) \notin U_{p,l'}$ where we do allow $l'$ to vary.  It follows from Lemma \ref{basicpropertiesQ}(i) that if $k = k'$, $l\geq l'$, then
$$  k'x + t'(z)\in U_{p, l'}  \ \text{ if and only if } \  t(z)-t'(z) \in U_{p, l'}. $$
So we have means to replace conjuncts of $\varphi(x,z)$ by  terms independent of the variable $x$. However, the above will not work if $k \neq k'$ or $l<l'$. By Lemma \ref{basicpropertiesQ}(iii), across models of $\WSFQ$,  we have that 
$$kx+ t(z) \in U_{p, l} \text{ if and only if } hkx + ht(z) \in U_{p, l+ v_p(h)} \ \ \text{ for all } h \neq 0.$$
From this observation, it is easy to see that we can resolve the issue of having $k \neq k'$, and moreover arrange that $l\geq 0$ which will be used in the next observation.
By Lemma \ref{basicpropertiesQ}(i,ii), across models of $\WSFQ$,  we have that 
$$  kx + t(z) \in U_{p, l}  \ \text{ if and only if } \ \bigvee_{i=1}^{p^m} kz + t(z) + ip^{l}  \in U_{p, l+m} \text{ for all } l\geq 0 \text{ and all } m. $$
Using the preceding two observations we resolve the issue of having $l<l'$. The statement of the lemma for this case then follows from the second paragraph.

We now prove the full lemma. It suffices to consider the case where $\varphi(x,z)$ is a conjunction of atomic formulas. In view of the preceding paragraph, we reduce further to the case where $\varphi(x,z)$ is of the form 
$$ \bigwedge_{i=1}^m kx+t_i(z) \notin U_{p,l_i}  $$
We now show that $\exists x\varphi(x,z)$ is a tautology over $\WSFQ$ and thus complete the proof. Suppose $(G; \sU^G, \sP^G) \models \WSFQ $ and $c \in G^n$. It suffices to find $a \in G$ such that the $p$-condition  $ka+t_i(c) \notin U^G_{p,l_i}$ holds for all $i \in \{1, \ldots, m\}$ . Without loss of generality, we assume that $t_1(c), \ldots, t_{m'}(c)$ are not in $U^G_{p,l}$ for all $l$ and that  $t_{m'+1}(c), \ldots, t_{m}(c)$ are in $U^G_{p,l_0}$ for some $l_0 $ such that $l_0 < l_i$ for all $i \in \{1, \ldots, m\}$. Using \ref{basicpropertiesQ}(ii), choose $a$ such that $ka \in U^G_{p,l_0-1} \setminus  U^G_{p,l_0}$. It follows from Lemma \ref{basicpropertiesQ}(i) that $a$ is as desired. 
\end{proof}

\begin{thm} \label{QE}
The theories $\TSFZ$, $\TSFQ$, and $\TOSFQ$ admit quantifier elimination.
\end{thm}
\begin{proof}
As the three situations are very similar, we will only present here the proof that $\TOSFQ$ admits quantifier elimination. The proof for  $\TSFZ$ and $\TSFQ$ are simpler as there is no ordering involved. Along the way we point out the necessary modifications needed to get the proof for  $\TSFZ$ and $\TSFQ$. Fix $\TOSFQ$-models  $(G; <, \sU^G, \sP^G)$ and $(H; <, \sU^H, \sP^H)$ such that the latter is $|G|^+$-saturated. Suppose 
$$f \text{ is a partial } L_{\mathrm{ou}}^*\text{-embedding from } (G; <, \sU^G, \sP^G) \text{ to } (H; <, \sU^H , \sP^H ), $$ in other words, $f$ is an $L_{\mathrm{ou}}^*$-embedding of an $L_{\mathrm{ou}}^*$-substructure of $(G; <, \sU^G, \sP^G)$ into $(H; <, \sU^H , \sP^H)$. 
By  a standard test, it suffices to show that  if $\text{Domain}(f) \neq G$, then there is a partial $L_{\mathrm{ou}}^*$-embedding from $(G; <, \sU^G, \sP^G)$  to $(H; <, \sU^H, \sP^H)$ which properly extends $f$. For the corresponding statements with $\TSFZ$ or $\TSFQ$, we need to consider instead $(G; \sU^G, \sP^G)$ and $(H; \sU^H, \sP^H)$  depending on the situation.

We remind the reader that our choice of language includes a symbol for additive inverse, and so $\text{Domain}(f)$ is automatically a subgroup of $G$. Suppose $\text{Domain}(f)$ is not a pure subgroup of $G$, that is, there is an element $\text{Domain}(f)$ which is $n$-divisible in $G$ but not $n$-divisible in $\text{Domain}(f)$ for some $n>0$.   Then there is prime $p$ and $a$ in $G \setminus \text{Domain}(f)$ such that $pa \in \text{Domain}(f) $. Using divisibility of $H$, we get $b \in H$  such that $pb = f(pa)$. Let $g$ be the extension of $f$ given by
$$ ka + a' \mapsto kb + f(a') \quad \text { for }  k\in \{ 1, \ldots, p-1\} \text{ and } a' \in \text{Domain}(f). $$
It is routine to check that $g$ is an ordered group isomorphism from $\la \text{Domain}(f), a\ra$ to $\la \text{Image}(f), b\ra$.  It is also easy to check using Lemma \ref{basicpropertiesQ}(iii) that $ka + a'  \in U^G_{p',l}$ if and only if $kb + f(a') \in U^G_{p',l}$ and $ka + a'  \in P^G_{m}$ if and only if $kb + f(a') \in U^G_{m}$ for all $k$, $l$, $m$, primes $p'$, and $a' \in \text{Domain}(f)$. Hence,
$$g \text{ is a partial } L_{\mathrm{ou}}^*\text{-embedding } \text{ from } (G; <, \sU^G, \sP^G) \text{ to } (H; <, \sU^H , \sP^H ). $$
Clearly, $g$ properly extends $f$, so the desired conclusion follows. The proof for $\TSFQ$ is the same but without the verification that the ordering is preserved. The situation for $\TSFZ$ is slightly different as $H$ is not divisible. However, for all primes $p'$, $p'a$ is in  $p'G = U^G_{p', 1}$, and so $f(p'a)$ is in $ U^H_{p', 1} = p'H$. The proof proceeds similarly using \ref{basicpropertiesZ}(4-6).

The remaining case is when $\text{Domain}(f) \neq G$ is a pure subgroup of $G$. Let $a$ be in $G \setminus \text{Domain}(f)$. We need to find $b$ in $H \setminus\text{Image}(f)$ such that $$\text{qftp}_{L_{\mathrm{ou}}^*}(a \slash \text{Domain}(f)) = \text{qftp}_{L_{\mathrm{ou}}^*}(b \slash \text{Image}(f)). $$ By the fact that $\text{Domain}(f)$ is pure in $G$, and Corollary \ref{SimplifytoG2}, $\text{qftp}_{L_{\mathrm{ou}}^*}(a \mid \text{Domain}(f))$ is isolated by formulas of the form
$$ \rho(b) \wedge \lambda(x,b) \wedge \psi(x, t(b), t'(b)) $$
where $\rho(y)$ is a quantifier-free $L^*_{\mathrm{ou}}$-formula, $\lambda(x,y)$ is an order condition,  $\psi(x, z, z')$ a special formula, $t(y)$ and $t'(y)$ are tuples of $L_{\mathrm{ou}}^*$-terms of suitable length,  $b$ is a tuple of elements of $\text{Domain}(f)$ of suitable length, and $\psi(x, t(b), t'(b))$ is a nontrival $\text{Domain}(f)$-system. As $\text{Domain}(f)$ is a pure subgroup of $G$, we can moreover arrange that $\lambda(x,b)$ is simply the formula $b_1 < x <b_2$. Since $f$ is an $L_{\mathrm{ou}}^*$-embedding, $\rho(f(b))$ holds,  $f(b_1) < f(b_2)$, and $\psi\big(x, t(f(b)), t'(f(b))\big) $ is a nontrivial $\text{Image}(f)$-system. Using the fact that $(H; <, \sU^H , \sP^H )$ is $|G|^+$-saturated, the problem reduces to showing that 
$$\psi\Big(x, f\big(t(b)\big), f\big(t'(b)\big)\Big) \ \text{ has a solution in the interval } (f(b_1), f(b_2))^H.$$ As $\psi(x, t(b), t'(b))$ is satisfiable in $G$, it is locally satisfiable in $G$ by Lemma \ref{GlobaltoLocal}. For each $p$, let $\psi_p(x, z, z')$ be the  associated $p$-condition of $\psi(x, z, z')$. By Lemma \ref{localQE}, for all $p$, the formula $\exists x \psi_p(x, z, z')$ is equivalent modulo $\WSFQ$ to a quantifier free formula in $L^*_{\mathrm{u}}(z, z')$. Hence, $\exists x \psi_p\Big(x, f\big(t(b)\big), f\big(t'(b)\big)\Big)$ holds in $(H;<, \sU^H, \sP^H)$ for all $p$. Thus, $$\text{the } \text{Image}(f)\text{-system }\psi\Big(x, f\big(t(b)\big), f\big(t'(b)\big)\Big)  \text{ is locally satisfiable in } H.$$ The desired conclusion follows from the genericity of $(H; <, \sU^H, \sP^H)$.  The proofs for $\TSFZ$ and $\TSFQ$ are similar. However,  we have there the formula  $\bigwedge_{i =1}^k x \neq b_i$ with $k \leq |b|$ instead of the formula $b_1< x< b_2$, Lemma \ref{SimplifytoG} instead of Corollary \ref{SimplifytoG2}, and the corresponding notion of genericity instead of the current one. 
\end{proof}

\begin{cor} \label{Completeness}
The theory $\TSFZ$ is a recursive axiomatization of $\mathrm{Th}(\zz; \sU^\zz, \sP^\zz)$, and is therefore decidable. Similar statements hold for  $\TSFQ$ in relation to $\mathrm{Th}(\qq; \sU^\qq, \sP^\qq)$ and $\TOSFQ$ in relation to $\mathrm{Th}(\qq; < \sU^\qq, \sP^\qq)$.
\end{cor}
\begin{proof}
By Lemma \ref{basicpropertiesZ}(ii), the  subgroup generated by $1$ in an arbitrary model $(G; \sU^G, \sP^G)$ of $\TSFZ$ is an isomorphic copy of $(\zz; \sU^\zz, \sP^\zz)$. Hence by Theorem \ref{QE}, $\TSFZ$ is complete, and  on the other hand $(\zz; \sU^\zz, \sP^\zz) \models \TSFZ $ by Theorem \ref{genericity}. The first statement of the corollary  follows.  The justification of the second statement is obtained in a similar fashion.
\end{proof}

\begin{proof}[Proof of Theorem \ref{structure1}, part 1] We show that  the $L_{\mathrm{u}}$-theory of $(\zz; \SF)$ is model complete and decidable. For all $p$, $l\geq 0$, $m>0$, and all $a \in \zz$, we have the following:
\begin{enumerate}
\item $a \in U^\zz_{p, l}$ if and only there is $b\in \zz$ such that $p^lb =a$;
\item $a \notin U^\zz_{p, l}$ if and only if for some $i \in \{1, \ldots, p^l-1\}$, there is $b\in \zz$ such that $p^lb =a+i$;
\item $a \in P^\zz_m$ if and only if for some $d \mid m$, there is $b\in \zz$ such that $a =bd$ and $b\in \SF$;
\item  $a \notin P^\zz_m$ if and only if for all $d \mid m$, either for some $i \in \{1, \ldots, d-1\}$, there is $b\in \zz$ such that $db =a+i$ or there is $b \in \zz$ such that $a =bd$ and $b \notin \SF$.
\end{enumerate}
As $(\zz; \sU^\zz, \sP^\zz) \models \TSFZ$, it then follows from Theorem \ref{QE} and the above observation that every $0$-definable set in $(\zz, \SF)$ is existentially $0$-definable. Hence, the theory of $(\zz; \SF)$ is model complete. The decidability of  $\text{Th}(\zz; \SF)$ is immediate from the preceding corollary.
\end{proof}

\begin{lem}
Suppose $a\in \qq$ has $v_p(a)<0$. Then there is $\varepsilon \in \qq$ such that $v_p(\varepsilon) \geq 0$ and $a+\varepsilon \in \SQ$.
\end{lem}
\begin{proof}
Suppose $a$ is as stated. If $a \in \SQ$ we can choose $\varepsilon =0$, so suppose $a$ is in $\qq \setminus \SQ$. We can also arrange that $a>0$. Then there are $m, n, k \in \nn^{\geq 1}$ such that 
$$a = \frac{m}{np^k},\  (m, n)=1,\  (m, p)=1, \text{ and } (n, p)=1.$$ 
It suffices to show there is $b \in \zz$ such that $m+ p^kb$ is a square-free integer as then
$$a+ \frac{b}{n} = \frac{m+ p^kb}{np^k} \in \SQ.$$
For all prime $l$,  $p^kb_l+m \notin U^\qq_{l,2}$ for $b_l=0$ or $1$.  The  conclusion then follows from the genericity of $( \zz; \sU^\zz, \sP^\zz)$ as established in Theorem \ref{genericity}.
\end{proof}

\begin{cor}\label{Uni}
For all $p$ and $l$, $U^\qq_{p, l}$ is universally $0$-definable in $(\qq, \SQ)$.
\end{cor}
\begin{proof}
We will instead show that $\qq\setminus U^\qq_{p, l} = \{ a : v_p(a) < l \}$ is existentially $0$-definable for all $p$ and $l$. As $\qq \setminus U^\qq_{p, l+n} = p^n( \qq \setminus U^\qq_{p, l})$ for all $p$, $l$, and $n$, it suffices to show the statement for  $l=0$. Fix a prime $p$. 
By the preceding lemma we have that  for all $ a$, $v_p(a)<0$ if and only if  
$$ \text{ there is } \varepsilon \text{ such that } v_p(\varepsilon) \geq 0,  
 a+\varepsilon \in \SQ \text{ and } v_p(a+\varepsilon) <0.$$ 
We recall that $\{\varepsilon : v_p(\varepsilon) \geq 0 \}$ is existentially $0$-definable by Lemma \ref{Definabilityofvaluation}. Also, for all $a' \in \SQ$, we have that $v_p(a')< 0$ is equivalent to $p^2a' \in \SQ$. The conclusion hence follows.
\end{proof}

\begin{proof}[Proof of Theorem 1.3 and 1.4, part 1] We show that the $L_{\mathrm{u}}$-theory of $(\qq; \SQ)$ and the $L_{\mathrm{ou}}$-theory of $(\qq; <, \SQ )$ are model complete and decidable. The proof is almost exactly the same as that of part 1 of Theorem 1.1. It follows from Lemma \ref{Definabilityofvaluation} and Corollary \ref{Uni} that for all $p$ and $l$, the sets $U^\qq_{p, l}$ are existentially and universally $0$-definable in $(\qq; \SQ)$.  For all $m$, $P^\qq_m = m\SQ$ and $\qq \setminus P^\qq_m = m( \qq \setminus \SQ)$ are clearly existentially $0$-definable. The conclusion follows. 
\end{proof}

\noindent Next, we will show that the $L_{\mathrm{ou}}$-theory of $(\zz; <, \SF)$ is bi-interpretable with arithmetic. The proof follow closely the arguments from \cite{BJW}. In fact, we can slightly modify Corollary \ref{Consecutivesquarefree}  to use essentially the same proof at the cost of replacing $n^2$ with $n^2+n$. 
\begin{lem}
Let $c_1, \ldots, c_n$ be an increasing sequence of natural numbers, assume that for all primes $p$, there is a solution to the system of congruence inequations
$$ x + c_i \notin U^\zz_{p, 2} \text{ for all } i \in \{1, \ldots, n\}.$$ Then there is $a \in \nn$ such that $a+c_1, \ldots, a+c_n$ are consecutive square-free integers.
\end{lem}
\begin{proof}
Suppose $c_1, \ldots, c_n$ are as given. Let $c'_1, \ldots, c'_{n'}$ be the listing in increasing order of elements in the set of  $c \in \nn$ such that $c_1 \leq c \leq c_n$ and $c \neq c_i$ for $i \in \{1, \ldots, n\}$. The conclusion that there are infinitely many $a$ such that 
$$  \bigwedge_{i=1}^n (a+c_i \in \SF) \wedge \bigwedge_{i=1}^{n'} (a+c'_i \notin \SF) $$
follows from the assumptions about $c_1, \ldots, c_n$ and the genericity of $( \zz; \sU^\zz, \sP^\zz)$ as established in Theorem \ref{genericity}.
\end{proof}
\begin{cor} \label{Consecutivesquarefree}
For all $n\in \nn^{>0}$, there is $a \in \nn$ such that $a+1, a+4, \ldots, a+n^2$ are consecutive square-free integers .
\end{cor}
\begin{proof}
For each $p$, we can obtain $a\in \{1,2,\ldots,p^2-1\} $ such that 
$$a\nequiv_{p^2}-m^2 \text{ for all } m.$$ Hence, for any given $n>0$ and $p$, the $p$-condition $\bigwedge_{i=1}^{n}(x+i^2\notin U^{\zz}_{p,2})$ has a solution. The result now follows immediately  from the preceding lemma.
\end{proof}

\begin{proof}[Proof of Theorem \ref{structure2}]
It suffices to show that $(\zz; <, \SF)$ interprets multiplication on $\nn$. 
Let $T$ be the set of $(a, b) \in \nn^2$ such that for some $n \in \nn^{\geq 1},$ 
$$ b= a+n^2 \text{ and } a+1, a+4, \ldots, a+n^2 \text{ are consecutive square-free integers}. $$
\noindent The set $T$ is definable in $(\zz; <, \SF)$ as $(a, b) \in T$ and $b\neq a+1$ if and only if  $a+4\leq b$, $a+1$ and $a+4$ are consecutive square-free integers, $b$ is square-free, and whenever $c$, $d$, and $e$ are consecutive square-free integers with $a< c <d< e \leq b$, we have that $$(e-d)-(d-c) =2.$$ Let $S$ be the set $\{ n^2 : n \in \nn\}$. If $c=0$ or there are $a, b$ such that $(a, b) \in T$ and $b-a =c$, then $c=n^2$ for some $n$. Conversely, if $c  =n^2$, then either $c=0$ or by Corollary \ref{Consecutivesquarefree}, 
$$ \text{ there is } (a,b) \in T \text{ with } b-a =c.$$ 
Therefore, $S$ is definable in  $(\zz; <, \SF)$. The map $n \mapsto n^2$ in $\nn$ is definable in $(\zz; <, \SF)$ as $b =a^2$ if and only if $b\in S$ and whenever $c\in S$ is such that $c>b$ and $b, c$ are consecutive in $S$, we have that $c-b = 2a+1$. Finally, $c =ba$ if and only if $2c= (b+a)^2-b^2-a^2$. Thus, multiplication on $\nn$ is definable in $(\zz; <, \SF)$.
\end{proof}
\section{Combinatorial Tameness}

\noindent As the theories $\TSFZ$, $\TSFQ$, and $\TOSFQ$ are complete, it is convenient to work in the so-called monster models, that is, models which are very saturated and homogeneous. Until the end of the paper, let $(\gg; \sU^\gg, \sP^\gg)$ be a monster model of either  $\TSFZ$ or $\TSFQ$ depending on the situation. In the latter case, we suppose $(\gg; <, \sU^\gg, \sP^\gg)$ is a monster model of $\TOSFQ$. We assume that $\kappa, A$ and $I$ have
small cardinalities compared to $\gg$.

\medskip
\noindent Our general strategy to prove the tameness of $\TSFZ$, $\TSFQ$, and $\TOSFQ$ is to link them to the corresponding ``local'' facts. The next lemma  says that  $\TSFZ$ is ``locally'' supersimple of U-rank $1$.

\begin{lem}\label{localUrank1}
Suppose $(\gg; \sU^\gg, \sP^\gg) \models \TSFZ$, $\theta_p(x, y)$ is a consistent $p$-condition, and $b$ is in $\gg^{|y|}$. Then $\theta_p(x, b)$ does not divide over any base set $A\subseteq \gg$.
\end{lem}

\begin{proof}
Recall that every every $p$-condition is equivalent modulo $\TSFZ$ to a formula in the language $L$ of groups, and the reduct of $\TSFZ$ to $L$ is simply $\text{Th}(\zz)$. Hence, the desired conclusion is an immediate consequence of the well-known fact that $\text{Th}(\zz)$ is superstable of $U$-rank $1$~\cite{Qo}; see for example .
\end{proof}

\begin{proof}[Proof of Theorem \ref{structure1}, part 2]
We first show that $\text{Th}(\zz; \SF)$  is supersimple of U-rank $1$;  see \cite[p. 36]{Kim} for a definition of U-rank or SU-rank. By  the fact that $(\zz; \SF)$ has the same definable sets as  $(\zz; \sU^\zz, \sP^\zz)$ and  Corollary \ref{Completeness}, we can replace $\text{Th}(\zz; \SF)$ with $\TSFZ$. Suppose $(\gg; \sU^\gg, \sP^\gg) \models \TSFZ$. Our job is to show that every $L^*_{\mathrm{u}}(\gg)$-formula $\varphi(x,b)$ which forks over a small subset $A$ of $\gg $ must define a finite set in  $\gg$.  We can easily reduce to the case that $\varphi(x,b)$ divides over $A$. Moreover, we can assume that $\varphi(x,b)$ is quantifier free by Theorem \ref{QE} which states that $(\gg; \sU^\gg, \sP^\gg)$ admits quantifier elimination. Using Lemma \ref{SimplifytoG}, we can also arrange that $\varphi(x, b)$ has the form 
$$ \rho(b) \wedge  \varepsilon(x, b) \wedge  \psi(x, t(b), t'(b))   $$
where  $\rho(y)$ is a quantifier-free formula, $\varepsilon(x,y)$ is an equational condition, $t(y)$ and $t'(y)$ are tuples of $L^*_{\mathrm{u}}$-terms with length $n$ and $n'$ respectively, and $\psi(x, z, z')$ is a special formula.

 Suppose to the contrary  that $\varphi(x,b)$ divides over $A$ but $\varphi(x,b)$ defines an infinite set in $\gg$.
From the first assumption, we get an infinite ordering $I$ and a family $(\sigma_i)_{i \in I}$ of $L^*_{\mathrm{u}}$-automorphisms of $(\gg; \sU^\gg, \sP^\gg)$ such that $(\sigma_i (b))_{i \in I}$ is indiscernible over $A$ and $\bigwedge_{i \in I} \varphi(x, \sigma_i (b))$ is inconsistent. As $\varphi(x,b)$ defines an infinite set in $\gg$,  we get from the second assumption that $\rho(b)$ holds in $\gg$,
$\varepsilon(x, b)$ defines a cofinite set in $\gg$, and $\psi(x,t(b), t'(b))$ defines an infinite hence non-empty set in $\gg$. As $(\sigma_i (b))_{i \in I}$ is indiscernible, we have that $\rho(\sigma_i (b))$ holds in $\gg$ and $\varepsilon(x, \sigma_i (b))$ defines a cofinite set in $\gg$ for all $i \in I$. Using the saturation of $\gg$, we get a finite set $\Delta \subseteq I$ such that  
$$ \theta_{\Delta}(x):= \bigwedge_{i \in \Delta}  \psi\big(x, t(\sigma_i (b)), t'(\sigma_i (b))\big)  \text{ defines a  finite set in } \gg. $$
As $\theta_{\Delta}(x)$ is a conjunction of $\gg$-systems given by the same special formula, it is easy to see that $\theta_{\Delta}(x)$ is also a $\gg$-system. 

We will show that $\theta_{\Delta}(x)$ defines an infinite set and thus obtain the desired contradiction. As $(\gg; \sU^\gg, \sP^\gg)$ is a model of $ \TSFZ$ and hence generic, it suffices to show that $\theta_{\Delta}(x)$ is non-trivial and locally satisfiable. As $\varphi(x,b)$ is consistent, $t(b)$ has no common components with $t'(b)$.
The assumption  that $(\sigma_i(b))_{i \in I}$ is indiscernible gives us that $t(\sigma_i(b))$ has no common components with $t'(\sigma_j (b))$ for all $i$ and $j$ in $I$. It follows that $\theta_{\Delta}(x)$ is non-trivial. For each $p$, let  $\psi_p(x, z, z')$  be the associated $p$-condition of $\psi(x,z, z')$. For all $p$, we have that $\psi_p(x, t(b), t(b'))$ defines a nonempty set and consequently by Lemma \ref{localUrank1}, 
$$ \bigwedge_{i \in \Delta}\psi_p\big(x, t(\sigma_i(b)), t'(\sigma_i(b))\big) \text{ defines a nonempty set in } \gg.  $$
We easily check that the above means $\theta_\Delta(x)$ is $p$-satisfiable for all $p$. Thus $\theta_{\Delta}(x)$  is locally satisfiable which completes our proof that $\text{Th}(\zz, \SF)$ has U-rank $1$.

We will next prove that $\text{Th}(\zz, \SF)$ is $k$-independent for all $k>0$;  see \cite{CPT} for a definition of $k$-independence. The proof is almost the exact replica of the proof in \cite{ShelahKaplan} except the necessary modifications taken in the current paragraph. Suppose $l>0$, $S$ is an arbitrary subset of $\{ 0, \ldots, l-1\}$. Our first step is to show that there are $a, d \in \nn$ such that for $t \in \{0, \ldots, l-1\}$, 
$$a +td \text{ is square-free }\  \text{ if and only if }\  t \text{ is in } S.$$  Let $n= |S|$ and $n' = l -n$, and let $c \in \zz^n$ be the increasing listing of elements in $S$ and $c' \in \zz^{n'}$ the increasing listing of elements in $\{0, \ldots, l-1\}\setminus S$. Choose $d = (l!)^2$. We need to find $a$ such that 
$$ \bigwedge_{i=1}^n (a+ c_id \in \SF) \wedge \bigwedge_{i=1}^{n'} (a+ c'_id \notin \SF). $$
For $p\leq l$, if $a_p \notin p^2\zz = U_{p,2}^\zz$, then $a_p + c_id \notin p^2\zz$ for all $i \in \{1, \ldots, n\}$. For $p>l$, it is  easy to see that $ 0+  c_id \notin p^2\zz$ for all  $i \in \{1, \ldots, n\}$. The desired conclusion follows from the genericity of $(\zz; \sU^\zz, \sP^\zz)$.

Fix $k>0$. We construct an explicit  $L_{\mathrm{u}}$-formula which witnesses the $k$-independence of $\text{Th}(\zz, \SF)$. Let  $y =(y_0, \ldots, y_{k-1})$ and let $\varphi(x, y)$ be a quantifier-free $L^*_{\mathrm{u}}$-formula such that for all $a \in \zz$ and $b \in \zz^k$,
$$ \varphi(a, b) \ \text{ if and only if }\ a+ b_0+\cdots+b_{k-1} \in \SF \quad\text{ where } b =(b_0, \ldots, b_{k-1}).  $$
We will show that for any given $n >0$, there are families $(a_\Delta)_{ \Delta \subseteq \{ 0, \ldots, n-1\}^k}$ and $(b_{ij})_{0\leq i < k, 0 \leq  j < n }$ of integers such that
$$ \varphi(a_\Delta, b_{0,j_0}, \ldots, b_{k-1,j_{k-1}} ) \ \text{ if and only if }\ (j_0, \ldots, j_{k-1}) \in \Delta. $$
Let $f: \sP(\{0,\ldots, n-1\}^k) \to \{0,\ldots, 2^{(n^k)}-1\}$ be an arbitrary bijection. Let $g$ be the bijection from $\{ 0, \ldots, n-1\}^k$ to $\{ 0, \ldots, n^k-1\}$ such that if $b$ and $b'$ are in $\{ 0, \ldots, n-1\}^k $ and $b<_{\text{lex}}b'$, then $g(b)<g(b')$. More explicitly, we have
$$g(j_0, \ldots, j_{k-1}) =j_0n^{k-1} + j_1n^{k-2} + \cdots + j_{k-1} \text{ for } (j_0, \ldots, j_{k-1}) \in \{ 0, \ldots, n-1\}^k. $$
It follows from the preceding paragraph that we can find an arithmetic progression $(c_i)_{ i \in  \{0, \ldots, n^k 2^{(n^k)} -1 \} }$ such that for all $\Delta \subseteq \{0, \ldots, n-1\}^k$ and $(j_0, \ldots, j_{k-1})$ in $\{ 0, \ldots, n-1\}^k$, we have that
$$  c_{f(\Delta)n^k+ g(j_0, \ldots, j_{k-1})} \in \SF\  \text{ if and only if } \ (j_0, \ldots, j_{k-1}) \in \Delta. $$
Suppose $d = c_1-c_0$.  Set $b_{ij} =  djn^{k-i-1} $ for $i \in \{0, \ldots, k-1\}$ and $j \in \{0, \ldots, n-1\}$, and set $a_{\Delta} = c_{f(\Delta)n^k}$ for $\Delta \subseteq \{ 0, \ldots, n-1\}^k$. We have
$$ c_{f(\Delta)n^k+ g(j_0, \ldots, j_{k-1})} = c_{f(\Delta)n^k} + d g(j_0, \ldots, j_{k-1}) = a_\Delta + b_{0,j_0}+ \cdots+ b_{k-1,j_{k-1}}.   $$
The conclusion thus follows.
\end{proof}

\begin{lem} \label{localstability}
Every $p$-condition $\theta_p(x,y)$ is stable in  $\TSFQ$.
\end{lem}
\begin{proof}
Suppose $\theta_p(x,y)$ is as in the statement of the lemma. It is clear that if $\theta_p(x,y)$ does not contain the variable $x$, then it is stable. As stability is preserved  under taking boolean combinations, we can reduce to the case where $\theta_p(x,y)$ is $kx+t(y) \in U_{p,l}$ with $k\neq 0$. We note that for any $b$ and $b'$ in $\gg^{|y|}$, the sets defined by $\theta_p(x, b)$  and $\theta_p(x, b')$ are either the same or disjoint. It follows easily that $\theta_p(x, y)$ does not have the order property; in other words, $\theta_p(x,y)$ is stable. Alternatively, the desired conclusion also follows from the fact that  $( \qq; \sU^\qq)$ is an abelian structure and hence stable; see \cite[p. 49]{Wagner} for the relevant definition and result.
\end{proof}

\begin{proof}[Proof of Theorem \ref{structure3}, part 2]
We first show that $\text{Th}(\qq; \SQ)$ is simple. By the fact that $(\qq; \SQ)$ has the same definable sets as $(\qq; \sU^\qq, \sP^\qq)$ and
Corollary \ref{Completeness}, we can replace $\text{Th}(\qq; \SQ)$ with $\TSFQ$. Towards a contradiction, suppose that the latter is not simple. We obtain a formula  $\varphi(x,y)$ witnessing the tree property of $\TSFQ$; see \cite[pp. 24-25]{Kim} for the definition and proof that this is one of the equivalent characterizations of simplicity.
We can arrange that $\varphi(x,y)$ is quantifier-free by  Theorem \ref{QE}. Recall that disjunction preserves simplicity of formulas; this can be shown directly as an exercise or can be seen immediately from the equivalence between (1) and (3) in  \cite[Lemma~2.4.1]{Kim}. Hence using Lemma \ref{SimplifytoG}, we can arrange that $\varphi(x,y)$ is of the form $$  \rho(y) \wedge \varepsilon(x, y)  \wedge \psi(x, t(y), t'(y))  $$
where  $\rho(y)$ is a quantifier-free $L^*_{\mathrm{u}}$-formula,   $\varepsilon(x, y)$ is an equational condition, $t(y)$ and $t'(y)$ are tuples of $L^*_{\mathrm{u}}$-terms with lengths $n$ and $n'$ respectively, and $\psi(x, z, z')$ is a special formula. Let $(\gg; \sU^\gg, \sP^\gg) \models \TSFQ$. Then there is $b \in \gg^k$ with $k=|y|$, an uncountable  cardinal $\kappa$, and a tree $(\sigma_s)_{ s \in \omega^{<\kappa} }$  of $L^*_{\mathrm{u}}$-automorphisms of $(\gg; \sU^\gg, \sP^\gg)$ with the following properties:
\begin{enumerate}
\item for all $s \in \omega^{< \kappa}$, $\{ \varphi(x, \sigma_{s\frown (i)}(b)) : i \in \omega   \}$ is inconsistent;
\item for all $\hat{s} \in \omega^\kappa$, $\{ \varphi(x, \sigma_{\hat{s} \res \alpha}(b)): \alpha <\kappa\}$ is consistent;
\item for every $\alpha<\kappa$ and $s, s' \in \omega^\alpha $, $ \text{tp}\big( (\sigma_{s\frown (i)}(b))_i\big)=\text{tp}\big( (\sigma_{s'\frown (i)}(b))_i\big)$.
\end{enumerate}
More precisely, we can get $b$, $\kappa$, and  $(\sigma_t)_{ t \in \omega^{<\kappa} }$ satisfying (1) and (2) from the fact that $\varphi(x,y)$ witnesses the tree property of $\TSFQ$, a standard Ramsey arguments, and the monstrosity of $(\gg; \sU^\gg, \sP^\gg)$.  We can then arrange that (3) also holds using results in \cite{Kimkimscow}; a direct argument is also straightforward.

We deduce the desired contradiction by showing that there is $ s \in \omega^{<\kappa}$ such that   $\{   \varphi(x, \sigma_{s\frown (i)}(b)) : i \in \omega\}$ is consistent. 
From  (1-3), we get for all $s \in \omega^{< \kappa}$ that  $\rho(\sigma_{s}(b) )$ holds  and $\varepsilon(x, \sigma_{s}(b))$ defines a cofinite set. By montrosity of $\gg$, it suffices to find $s \in  \omega^{<\kappa}$ such that  
  any finite conjunction of  $\{ \psi\big(x, t(\sigma_{s\frown (i)}(b)),t'(\sigma_{s\frown (i)}(b))\big) : i \in \omega\}$ defines an infinite set in $\gg$. For $s \in \omega^{<\kappa}$ and a finite $\Delta \subseteq \omega$, set 
  $$ \theta_{s, \Delta}(x) := \bigwedge_{i \in \Delta}  \psi\big(x, t(\sigma_{s\frown (i)}(b)),t'(\sigma_{s\frown (i)}(b))\big).  $$
 As $\kappa$ is uncountable, to ensure the desired $s\in \omega^{<\kappa}$ exists, it suffices to show for fixed $\Delta$  that for all but countably many $\alpha < \kappa$ and all $s \in \omega^\alpha$, the formula $\theta_{s, \Delta}(x)$ defines an infinite set in $\gg$. 
 
 Note that   $\theta_{s, \Delta}(x)$ is a conjunction of $\gg$-systems given by the same special formula, so $\theta_{s, \Delta}(x)$ is also a $\gg$-system.  By the genericity of $\TSFQ$ established in Theorem \ref{genericity}, we need to check that for all but countably many $\alpha < \kappa$ and  all $s \in \omega^\alpha$, the $\gg$-system  $\theta_{s, \Delta}(x)$ is nontrivial and locally satisfiable. Indeed, this implies that 
 By (2), $\varphi(x, b)$ is consistent, and so is  $\psi(x, t(b),t'(b))$. This implies in particular that $t(b)$ and $t'(b)$ have no common components.  It then follows from (3) that for $s \in \omega^{< \kappa}$ and $i, j \in \omega$, 
  $$t(\sigma_{s\frown (i)}(b)) \text{ and } t'(\sigma_{s\frown (j)} (b)) \text{  have no common elements }.$$ 
  Hence, $\theta_{s, \Delta}(x)$ is nontrivial for all $s \in \omega^{< \kappa}$. Let $\psi_p(x,z,z')$ be the associated $p$-condition of $\psi(x,z,z')$. We then get from (2) that $\{ \psi_p(x, t(\sigma_{\hat{s} \res \alpha}(b)), t'(\sigma_{\hat{s} \res \alpha}(b))): \alpha <\kappa\}$ is consistent  for all $\hat{s} \in \omega^\kappa$. By Lemma \ref{localstability}, the formula $\psi_p(x, t(y),t'(y)  )$ is stable and hence does not witness the tree property.  It follows that for all but finitely many $\alpha< \kappa$ and all $s \in \omega^\alpha$, the set 
  $$\{ \psi_p\big(x, t(\sigma_{s\frown (i)} (b)), t'(\sigma_{s\frown (i)} (b))\big) : i \in \omega\} \text{ is consistent}.$$ For such $s$, we have that $\theta_{s, \Delta}(x)$ is $p$-satisfiable. So for all but countably many $\alpha< \kappa$ and all $s \in \omega^\alpha$, $\theta_{s, \Delta}(x)$  is locally satisfiable which completes the proof that $\text{Th}(\qq; \SQ)$  is simple.

We next prove that $\text{Th}( \qq; \SQ)$ is not strong which implies that it is not supersimple;  for the definition of strength and the relation to supersimplicity see \cite{Adler}. Again, we can replace $\text{Th}( \qq; \SQ)$ by $\TSFQ$ using Proposition \ref{notintroducingnewdefinableset} and Corollary \ref{Completeness}.  For each $p$, let $\varphi_p(x,y)$ with $|y|=1$ be the formula $x-y \in U_{p,0}$.   
For all $p$ and $i$, set $b_{p,i} = p^{-i}$. We will show that $\big( \varphi_p(x,y),  ( b_{p,i})_{i \in \nn}   )     \big)$ forms an inp-pattern of infinite depth in $(\qq; \sU^\qq, \sP^\qq)$.  For distinct $i$ and $j$ in $\nn$, we have that $p^{-i} - p^{-j} \notin U^\qq_{p,0}$ which implies that $\varphi_p(x, b_{p,i}) \wedge \varphi_p(x, b_{p,j}) $ is inconsistent. On the other hand, if $S$ is a finite set of primes, and $f: S \to \nn$ is an arbitrary function, then for $a=\Sigma_{p\in S}b_{p,f(p)}$ we have that $(\qq; \sU^\qq, \sP^\qq)\models \bigwedge_{p \in S} \varphi_p(a, b_{p,f(p)})$. The desired conclusion follows.

Finally, we note that $(\zz; \sU^\zz, \sP^\zz)$ is a substructure of  $(\qq; \sU^\qq, \sP^\qq)$, the former theory admits quantifier elimination and has $\text{IP}_k$ for all $k>0$.  Therefore, the latter also has $\text{IP}_k$ for all $k>0$. In fact, the construction in part 2 of the proof of Theorem \ref{structure1} carries through.
\end{proof}

\begin{lem}\label{NIPoforderformula}
Any order-condition has $\mathrm{NIP}$ in $\TOSFQ$.
\end{lem}
\begin{proof}
The statement immediately follows from the fact that every order condition is a formula in the language of ordered groups and the fact that the reduct of any model of  $\TOSFQ$ to this language is an ordered abelian group, which has $\text{NIP}$; see for example \cite{Gurevich}.
\end{proof}

\begin{proof}[Proof of Theorem \ref{structure4}, part 2]
In the proof of part 2 of Theorem \ref{structure3}, we have shown that $\text{Th}(\qq; \SQ)$ is not strong and is $k$-independent for all $k>0$, so the corresponding conclusions for $\text{Th}(\qq; <, \SQ)$ also follow. 
It remains to show that $\text{Th}(\qq;<, \SQ)$ has $\text{NTP}_2$. The proof is essentially the same as the proof that $\text{Th}(\qq; \SQ)$ is simple, but with extra complications coming from the ordering.
By Proposition \ref{notintroducingnewdefinableset} and  Corollary \ref{Completeness}, we can replace $\text{Th}(\qq;<, \SQ)$ with $\OSFQ^*$.  Towards a contradiction, assume that there is a formula $\varphi(x,y)$ witnessing $\text{TP}_2$ (see \cite[pp. 700-701]{Chernikov}).
We can arrange that $\varphi(x,y)$ is quantifier-free by  Theorem \ref{QE}. Disjunctions of formulas with $\text{NTP}_2$ again have $\text{NTP}_2$\cite[p. 701]{Chernikov},  so using Lemma \ref{SimplifytoG2} we can arrange that $\varphi(x,y)$ is of the form
$$  \rho(y) 
\wedge \lambda(x,y) \wedge \psi(x, t(y), t'(y))  $$
where 
$\rho(y)$ is a quantifier-free $L^*_{\mathrm{ou}}$-formula$, \lambda(x, y)$ an order condition, $\psi(x, z, z')$ a special formula, and $t(y)$ and $t'(y)$ are tuples of $L^*_{\mathrm{ou}}$-terms with length $n$ and $n'$ respectively.
Then there is $b \in \gg^{k}$ with $k=|y|$ and an array $(\sigma_{ij})_{ i \in \omega, j \in \omega}$  of $L_{\mathrm{ou}}^*$-automorphisms of $(\gg; <, \sU^\gg, \sP^\gg)$ with the following properties:
\begin{enumerate}
\item for all $i \in \omega$, $\{ \varphi(x, \sigma_{ij}(b)) : j \in \omega   \}$ is inconsistent;
\item for all $f: \omega \to \omega$, $\{ \varphi(x, \sigma_{if(i)}(b)): i \in \omega\}$ is consistent;
\item for all $i \in \omega$,  $(\sigma_{ij}(b))_{j \in \omega}$ is indiscernible over  $\{ \sigma_{i'j}(b) : i' \in \omega, i'\neq i, j \in \omega  \}$;
\item the sequence of ``rows'' $( (\sigma_{ij}(b))_{j \in \omega})_{i \in \omega}$ is indiscernible.
\end{enumerate}
We could get $b$, $\omega$, and $(\sigma_{ij})_{ i \in \omega, j \in \omega}$ as above  from the definition of $\text{NTP}_2$, Ramsey arguments, and the monstrosity of $(\gg; \sU^\gg, \sP^\gg)$; see also  \cite[p. 697]{Chernikov} for the type of argument we need to get  (3).

We deduce that the  set $\{ \varphi(x, \sigma_{ij}b) : j \in \omega   \}$ is consistent for all $i \in \omega$, which is the desired contradiction. We get from (2) that $\rho(\sigma_{ij}b)$ holds for all $i \in \omega$ and $j \in \omega$. 
Hence, it suffices to show for all $i \in \omega$  that $$   \{   \lambda(x,\sigma_{ij}b) \wedge  \psi(x, t( \sigma_{ij}b), t'(\sigma_{ij} b)) : j \in \omega\}  \text{ is consistent}.$$ The order condition $\lambda(x, y)$ has $\text{NIP}$  by Lemma \ref{NIPoforderformula}, and so it has  $\text{NTP}_2$. Using conditions (2-4), we get that 
$$\{ \lambda(x, \sigma_{ij}(b)) : j \in \omega\} \text{ is consistent} \text{ for all } i\in \omega.$$ 
Hence, any finite conjunction from $\{\lambda(x, \sigma_{ij}(b)): j \in \omega\}$ contains an open interval for all $i \in \omega$. For $i \in \omega$ and a finite $\Delta \subseteq \omega$, set 
  $$ \theta_{i, \Delta}(x) := \bigwedge_{j \in \Delta}  \psi\big(x, t(\sigma_{ij}(b)),t'(\sigma_{ij}(b))\big).  $$
It suffices to show that $\theta_{i, \Delta}(x)$ defines a non-empty set in every non-empty $\gg$-interval.

We have that $\theta_{i, \Delta}(x)$ is a conjunction of $\gg$-system given by the same special formula, and so is again a $\gg$-system.
 By the genericity of $\TOSFQ$,  the problem reduces to showing $\theta_{i, \Delta}(x)$ is nontrivial and locally satisfiable.   
 By (2), $\varphi(x, b)$ is consistent, and so is  $\psi(x, t(b),t'(b))$. This implies in particular that $t(b)$ and $t'(b)$ have no common components.  It then follows from (3) that for $i \in \omega$ and distinct $j, j' \in \omega$, 
  $$t(\sigma_{ij}(b)) \text{ and } t'(\sigma_{ij'} (b)) \text{  have no common elements}.$$ Hence, $\theta_{i, \Delta}(x)$ is nontrivial for all $i \in \omega$. Let $\psi_p(x,z,z')$ be the associated $p$-condition of $\psi(x,z,z')$. We then get from (2) that $\{ \psi_p(x, \sigma_{if(i)}(b)): i \in \omega\}$ is consistent  for all $f: \omega \to \omega$. By Lemma \ref{localstability}, the formula $\psi_p(x, t(y),t'(y)  )$ is stable and hence has $\text{NTP}_2$.  It follows that for all but finitely many $i \in \omega$ the set 
  $$\{ \psi_p\big(x, t(\sigma_{ij} (b)), t'(\sigma_{ij} (b))\big) : j \in \omega\} \text{ is consistent}.$$ Combining with (4), we get that $\theta_{i, \Delta}(x)$ is  $p$-satisfiable for all $p$ which completes the proof.
  \end{proof}

\begin{cor}
The set $\zz$ is not definable in $(\qq; <, \SQ)$.
\end{cor}
\begin{proof}
Towards a contradiction, suppose $\zz$ is definable in $(\qq; <, \SQ)$. Then by Theorem \ref{structure2},  $(\nn; +, \times, <, 0, 1)$ is interpretable in $(\qq; <, \SQ)$. It then follows from Theorem \ref{structure4} that  $(\nn; +, \times, <, 0, 1)$ has  $\text{NTP}_2$, but this is well-known to be false.
\end{proof}

\section{Further questions}
\noindent There are several further questions we can ask about $(\zz; \SF)$, $(\qq; \SQ)$, and  $(\qq; <, \SQ)$. We would like to better understand dividing and forking inside these structures. Ideally, they coincide and have appropriate ``local to global'' behaviors. It would also be nice to understand imaginaries and definable groups in these structures.

One would like to have similar results for ``sufficiently random'' subsets of $\zz$ other than  $\text{Pr}$ and $\SF$. Another interesting candidate of such a subset is $\{\pm pq : p, q \text{ are primes}\}$. Most likely, it is not possible to prove the analogous results without assuming any number-theoretic conjecture. In a rather different direction, is there any sense in which we can say that most subsets of $\zz$ are ``sufficiently random'' and yield results similar to ours?

In \cite{BJW}, it is shown under the assumption of Dickson's Conjecture, that the monadic second order theory of $( \nn; S, \text{Pr})$ is decidable where $S$ is the successor function. We hope the analogous result for $( \nn; S, \SF)$ can be shown without assuming any conjecture. On another note, suppose the field  $\bar{\qq}$ is an algebraic closure of the field $\qq$, $v$ range over the non-archimedian valuations of $\bar{\qq}$, and $$ \text{Sqf}^{ \bar{\qq} } = \{ a \in \bar{\qq} :  v(a) <2 \text{ for all } v \}.$$
Does $( \bar{\qq}; \text{Sqf}^{ \bar{\qq} })$ have $\text{NTP}_2$? Finally, if $\zz^\times$ is the multiplicative monoid of integers, can anything be said about $(\zz^\times; \SF)$? 

\section*{Acknowledgements}
\noindent The authors would like to thank  Utkarsh Agrawal, William Balderrama, Alexander Dunn,  Lou van den Dries, Allen Gehret, Itay Kaplan, and Erik Walsberg  for dicussions and comments  at various stages of the project. We would also like to thank the referee for the detailed and insightful reading which allowed us to simplify several arguments. Of course, any remaining errors are our responsibility.
\bibliographystyle{amsalpha}
\bibliography{the}

\providecommand{\bysame}{\leavevmode\hbox to3em{\hrulefill}\thinspace}
\providecommand{\MR}{\relax\ifhmode\unskip\space\fi MR }
\providecommand{\MRhref}[2]{%
  \href{http://www.ams.org/mathscinet-getitem?mr=#1}{#2}
}
\providecommand{\href}[2]{#2}
\begin{thebibliography}{{Adl}07}

\bibitem[{Adl}07]{Adler}
H.~{Adler}, \emph{{Strong theories, burden, and weight}}, Preprint (2007).

\bibitem[BJW93]{BJW}
P.~T. Bateman, C.~G. Jockusch, and A.~R. Woods, \emph{Decidability and
  undecidability of theories with a predicate for the primes}, J. Symbolic
  Logic \textbf{58} (1993), no.~2, 672--687.

\bibitem[BPW00]{Qo}
Oleg Belegradek, Ya'acov Peterzil, and Frank Wagner, \emph{Quasi-o-minimal
  structures}, J. Symbolic Logic \textbf{65} (2000), no.~3, 1115--1132.
  \MR{1791366}

\bibitem[Che14]{Chernikov}
Artem Chernikov, \emph{Theories without the tree property of the second kind},
  Ann. Pure Appl. Logic \textbf{165} (2014), no.~2, 695--723.

\bibitem[Con18]{Gabe}
Gabriel Conant, \emph{There are no intermediate structures between the group of
  integers and {P}resburger arithmetic}, J. Symb. Log. \textbf{83} (2018),
  no.~1, 187--207. \MR{3796282}

\bibitem[CPT14]{CPT}
A.~{Chernikov}, D.~{Palacin}, and K.~{Takeuchi}, \emph{{On n-dependence}},
  ArXiv e-prints (2014).

\bibitem[DG17]{DolichGoodrick}
Alfred Dolich and John Goodrick, \emph{Strong theories of ordered {A}belian
  groups}, Fund. Math. \textbf{236} (2017), no.~3, 269--296.

\bibitem[GS84]{Gurevich}
Y.~Gurevich and P.~H. Schmitt, \emph{The theory of ordered abelian groups does
  not have the independence property}, Trans. Amer. Math. Soc. \textbf{284}
  (1984), no.~1, 171--182. \MR{742419}

\bibitem[Kim14]{Kim}
Byunghan Kim, \emph{Simplicity theory}, Oxford Logic Guides, vol.~53, Oxford
  University Press, Oxford, 2014. \MR{3156332}

\bibitem[KKS14]{Kimkimscow}
Byunghan Kim, Hyeung-Joon Kim, and Lynn Scow, \emph{Tree indiscernibilities,
  revisited}, Arch. Math. Logic \textbf{53} (2014), no.~1-2, 211--232.
  \MR{3151406}

\bibitem[KS16]{ShelahKaplan}
I.~{Kaplan} and S.~{Shelah}, \emph{{Decidability and classification of the
  theory of integers with primes}}, ArXiv e-prints (2016).

\bibitem[Mir47]{Mirsky}
L.~Mirsky, \emph{Note on an asymptotic formula connected with {$r$}-free
  integers}, Quart. J. Math., Oxford Ser. \textbf{18} (1947), 178--182.

\bibitem[Rog64]{Schrinelmanndensity}
Kenneth Rogers, \emph{The {S}chnirelmann density of the squarefree integers},
  Proc. Amer. Math. Soc. \textbf{15} (1964), 515--516.

\bibitem[Tra17]{Minh}
Minh~Chieu Tran, \emph{{Tame structures via multiplicative character sums on
  varieties over finite fields}}, ArXiv e-prints (2017).

\bibitem[Wag97]{Wagner}
Frank~O. Wagner, \emph{Stable groups}, London Mathematical Society Lecture Note
  Series, vol. 240, Cambridge University Press, Cambridge, 1997. \MR{1473226}

\end{thebibliography}

\end{document}